\documentclass[12pt]{amsart}
\usepackage{times}
\usepackage{anysize}
\marginsize{2.7cm}{2.44cm}{2.0cm}{3.0cm}
\usepackage[small,it]{caption}
\usepackage[OT2,T1]{fontenc}
\DeclareSymbolFont{cyrletters}{OT2}{wncyr}{m}{n}
\DeclareMathSymbol{\Sha}{\mathalpha}{cyrletters}{"58}
 \usepackage[latin1]{inputenc}
 \usepackage[dvips]{graphicx}
 \usepackage{wrapfig}
 \usepackage{amsmath}
 \usepackage{amsthm}
 \usepackage{amsfonts}
 \usepackage{amssymb}
 \usepackage{layout
} \usepackage{verbatim}
 \usepackage{alltt}
\usepackage{yfonts}

\usepackage[all]{xy}
\usepackage{setspace}

\newtheorem*{thma}{Theorem A}
\newtheorem*{thmb}{Theorem B}
\newtheorem*{thmc}{Theorem C}

\newcommand{\LL}{\Lambda}
\newcommand{\TT}{\mathbb{T}}
\newcommand{\QQ}{\mathbb{Q}}
\newcommand{\FF}{\mathcal{F}}

\newcommand{\lra}{\longrightarrow}
\newcommand{\ZZ}{\mathbb{Z}}
\newcommand{\PP}{\mathcal{P}}

\newcommand{\KKK}{\mathcal{K}}

\newcommand{\Gal}{\textup{Gal}}
\newcommand{\KS}{\textbf{\textup{KS}}}
\newcommand{\ES}{\textbf{\textup{ES}}}
\newcommand{\NN}{\mathcal{N}}
\newcommand{\ra}{\rightarrow}

\newcommand{\be}{\begin{equation}}
\newcommand{\ee}{\end{equation}}

\newcommand{\kk}{\frak{E}}
\newcommand{\al}{\mathcal{L}}

\newcommand{\oo}{\mathcal{O}}
\newcommand{\RR}{\mathcal{R}}
\newcommand{\eee}{\epsilon}

\newcommand{\FFc}{\mathcal{F}_{\textup{\lowercase{can}}}}

\newcommand{\hne}{\mathbf{H.nE}}

\newcommand{\ooo}{\frak{O}}

\numberwithin{equation}{section}
\newtheorem{thm}{Theorem}[section]
\newtheorem{lemma}[thm]{Lemma}
\newtheorem{conjecture}[thm]{Conjecture}
\newenvironment{define}{\par\medskip\noindent\refstepcounter{thm}
\bgroup{\hspace*{-0.15 cm}\bf{Definition}
\thethm.}\bgroup}{\egroup \egroup\par\medskip}\newtheorem{prop}[thm]{Proposition}
\newtheorem{cor}[thm]{Corollary}
\newenvironment{rem}{\par\medskip\noindent\refstepcounter{thm}
\bgroup{\hspace*{-0.15 cm}\bf{Remark} \thethm.}\bgroup}{\egroup
\egroup\par\medskip} \parskip 2pt

\newenvironment{example}{\par\medskip\noindent\refstepcounter{thm}
\bgroup{\hspace*{-0.15 cm}\bf{Example}
\thethm.}\bgroup}{\egroup \egroup\par\medskip}

\newtheorem{conj}{Conjecture}

\newcounter{Athm}[section]

\renewcommand{\theAthm} {\arabic{Athm}}

\begin{document}
%\title{{O}{\lowercase{n}
\title{{M}\lowercase{ain conjectures for} CM \lowercase{fields and a} Y\lowercase{ager-type theorem for}  R\lowercase{ubin}-{S}\lowercase{tark} \lowercase{elements}}

\author{K\^az\i m B\"uy\"ukboduk}

\email{kazim@math.stanford.edu}
%\address{K\^az\i m B\"uy\"ukboduk \hfill\break\indent Ko\c{c} University, Mathematics  \hfill\break\indent Rumeli Feneri Yolu, 34450 Sariyer \hfill\break\indent Istanbul, Turkey
%}
%\curraddr{  \hfill\break\indent Ko\c{c} University,  Mathematics \hfill\break\indent Rumeli Feneri Yolu \hfill\break\indent 34450 Sar\i yer/\.Istanbul
%\hfill\break\indent Turkey}
%\address{\textit{Current Address:} \hfill\break\indent
%Ih\'es, Le Bois-Marie, 35,  \hfill\break\indent Route de Chartres \hfill\break\indent F-91440 Bures-sur-Yvette
%\hfill\break\indent France}
%\keywords{Stark conjectures, Euler systems, Kolyvagin systems.}

\begin{abstract}
In this article, we study the $p$-ordinary Iwasawa theory of the (conjectural) Rubin-Stark elements defined over abelian extensions of a CM field $F$ and develop a rank-$g$ Euler/Kolyvagin system  machinery (where $2g=[F:\QQ]$), refining and generalizing Perrin-Riou's theory and the author's prior work.  This has several important arithmetic consequences: Using the recent results of Hida and Hsieh on the CM main conjectures, we prove a natural extension of a theorem of Yager for the CM field $F$, where we relate the Rubin-Stark elements to the several-variable Katz $p$-adic $L$-function. Furthermore, beyond the cases covered by Hida and Hsieh, we are able to reduce the $p$-ordinary CM main conjectures to a local statement about the Rubin-Stark elements. We discuss applications of our results in the arithmetic of CM abelian varieties.
\end{abstract}

\maketitle
\tableofcontents
\section{Introduction}
\label{sec:intro}

The main objective of this paper is to study Iwasawa's main conjecture for CM fields (see \cite[Page 90]{ht94} for its precise statement). When the CM field in question is an imaginary quadratic number field, the main conjecture was proved by Rubin~\cite{rubinmainconj} using the elliptic unit Euler system. For more general CM fields, it has been tackled by Hida and Tilouine~\cite{ht94} and Hida~\cite{hidaquadratic} along the anticyclotmic tower, by Mainardi in~\cite{mainardi} along the cyclotomic tower and by Hsieh in~\cite{hsiehCMmainconj} in more general context. All these works relied on the Eisenstein ideal method. In this paper, we approach this problem via the conjectural Rubin-Stark elements, by means of extending and refining the rank-$r$ Euler/Kolyvagin system machinery developed by the author in \cite{kbbesrankr}. Before stating our results in detail, we provide an overview of what we carry through in this paper:
\begin{enumerate}
\item[1.] Applying the general machinery we set up in this article with the Rubin-Stark elements (along with the aid of Nekov\'a\v{r}'s \cite{nekovar06} descent formalism incorporated in his theory of Selmer complexes), we prove Theorem A below where we show that the CM main conjecture is equivalent to a purely local statement about the Rubin-Stark elements. This local statement is labeled by Conjecture 2 below and %is essentially a weak form of a generalized Coates-Wiles explicit reciprocity law~\cite{coateswiles77, wiles78reciprocity} and
asserts that the collection of Rubin-Stark elements along the maximal $\ZZ_p$-power extension of a CM field recovers the Katz $p$-adic $L$-function.  \\
\item[2.] In particular, combining with a recent work of Hsieh~\cite{hsiehCMmainconj} in which he proves the CM main conjectures  under certain hypotheses, we are able to prove Conjecture 2 in many interesting cases.  Our second main result (that may be found in its precise form in Theorem B below) is a natural and a far reaching generalization of an important theorem of Yager~\cite[Theorem 1]{yager}, which in its original form relates the tower of elliptic units along the $\ZZ_p^2$-extension of an imaginary quadratic field to the two-variable $p$-adic $L$-function of Katz.\\
\item[3.] Assuming the truth of Conjecture~2 (beyond the cases covered by Theorem~B), we provide an application of Theorem~A in the arithmetic study of CM abelian varieties and to that end, we prove our third result, which is labeled as Theorem~C below. The only results so far that parallel Theorem~C are limited to the case when the abelian variety $\mathcal{A}$ in question is of $\textup{GL}_2$-type and is defined over $\QQ$; or else to the case $\dim \mathcal{A}=1$, i.e., $\mathcal{A}$ is an elliptic curve. That is why, albeit being conditional, we believe that Theorem~C should still be of interest.
\end{enumerate}

We remark that although the existence Rubin-Stark elements is conjectural, one may prove (as in \cite{kbb,kbbdeform}) that the Kolyvagin systems which we obtain using the Rubin-Stark elements and which play a crucial role in the proofs of our main theorems do exist unconditionally. See also Remark~\ref{rem:existenceofKS} below.

It is important to note that not only our approach significantly differs from that of \cite{ht94,hidaquadratic,mainardi,hsiehCMmainconj}, it also sheds light on the CM main conjectures well-beyond the cases covered in all previous work on this subject. To give an example, suppose $\mathcal{E}$ is an elliptic curve defined over a totally real field $K$ which has CM by the imaginary quadratic field $M$ and set $F=MK$. All prior results (in this level of generality, namely when $K$ is a general totally real field) concerning the main conjectures for $\mathcal{E}$ required the assumption that the sign $W(\mathcal{E}/K)$ of the functional equation of the Hasse-Weil $L$-function of $\mathcal{E}_{/K}$ is $+1$. Our Theorem A allows us to go beyond that (of course, it unfortunately proves less than the main conjectures, too).

The method that we develop in this article seems flexible enough to study the \emph{anticyclotomic main conjectures} in the following form, which did not seem possible with alternative techniques in the literature. Let $M_\infty^-$ denote the anticyclotomic $\ZZ_p$-extension of $M$ and suppose $\mathcal{E}/K$ is an elliptic curve as above. When $W(\mathcal{E}/K)=-1$, the Iwasawa theory of $\mathcal{E}$ along $F_\infty^-:=FM_\infty^-$ is much different from the cyclotomic Iwasawa theory. The relevant Selmer group in this case is not expected to be torsion and the methods of \cite{ht94,hidaquadratic,mainardi,hsiehCMmainconj} does not seem adequate to address this important case. We will use the techniques we develop in this paper in our forthcoming work~\cite{kbbanticycloK} to study the Iwasawa theory of $\mathcal{E}_{/K}$ along $F_\infty^{-}$, when $K$ is a general totally real field\footnote{When $K=\QQ$, this has been studied by Agboola and Howard in \cite{agboolahoward} using the elliptic unit Euler system; see also \cite{arnoldCMform} for a generalization of their results to higher weight forms, as well as \cite{arnoldhida} for the treatment of an analogous problem in the non-CM case, this time making use of Kato-Ochiai Euler system. } and when the sign of the functional equation is \emph{not} necessarily $+1$.

As in \cite{ht94,hidaquadratic,mainardi,hsiehCMmainconj}, the $p$-ordinary condition (labeled as (\textbf{H.pOrd}) below) plays a vital role also in this article. To best of our knowledge, nothing substantial has been proved in the absence of this hypothesis\footnote{Only when the totally real field $K$ is $\QQ$, Rubin~\cite{rubinmainconj} has proved a form of the two-variable main conjecture in the supersingular setting.}. In our forthcoming work~\cite{kbbstarksupersingular}, we adapt the approach of this article with the conjectural Rubin-Stark elements (and appropriately modify the rank-$r$ Euler/Kolyvagin system machinery that we develop here) in order to prove (a form of) the main conjectures of Iwasawa theory for CM fields at supersingular primes.  This falls away from the scope of all the previous work on the CM main conjectures which are alluded to above.

Before giving a further account of our results, we set our notation that will be in effect throughout this paper. Let $F$ be a CM field and $K$ its maximal totally real subfield that has degree $g$ over $\QQ$. Fix a complex conjugation $c \in \textup{Gal}(\overline{\QQ}/K)$ lifting the generator of $\textup{Gal}(F/K)$. Fix forever an odd prime $p$ unramified in $F/\QQ$ and an embedding $\iota_p: \overline{\QQ} \hookrightarrow \overline{\QQ}_p$.
\begin{define} \emph{A CM-type} $\Sigma$ is a collection of embeddings $\sigma: F \hookrightarrow \overline{\QQ}$ such that $\Sigma \cup {\Sigma}^c$ is the full set of embeddings of $F$ into $\overline{\QQ}$ and $\Sigma\cap\Sigma^c=\emptyset$. Here the complex conjugation $\sigma^c$ of an embedding $\sigma$ is defined so that $\sigma^c(a)=\sigma(a^c)$ for $a \in F$.
\end{define}
\noindent Suppose that the CM-type $\Sigma$ satisfies the following hypothesis:
\\\textbf{(H.pOrd)} The embeddings $\Sigma_p:=\{\iota_p\circ \sigma\}_{\sigma \in \Sigma}$ induce exactly half of the places of $F$ over $p$.
\\\\
Such a CM-type is called $p$-ordinary. We identify $\Sigma_p$ with the associated subset of primes $\{\wp_1,\cdots,\wp_s\}$ of $F$ above $p$. Setting $\Sigma_p^c=\{\wp_1^c,\cdots\wp_s^c\}$, we see that the disjoint union $\Sigma_p \sqcup \Sigma_p^c$ is the set of all primes of $F$ above $p$. As explained in \cite{hidaquadratic}, there then exists an abelian variety with CM by $F$ having good ordinary reduction at $p$, and its CM-type is $\Sigma$.

Let $F_\infty$ be the maximal $\ZZ_p$-power extension of $F$ and let $\Gamma=\textup{Gal}(F_\infty/F)$. Then $\Gamma\cong\ZZ_p^{g+1+\delta}$, where $\delta$ is the Leopoldt defect. Fix a finite extension $\frak{F}$ of $\QQ_p$ and denote its ring of integers by $\frak{O}$. Let $\varpi \in \ooo$ be a fixed uniformizing element. We let $\mathcal{W}$ be the valuation ring of $\widehat{\overline{\QQ}}_p$, the completion of $\overline{\QQ}_p$.
Set $\LL=\ooo[[\Gamma]]$ to be the Iwasawa algebra, which is isomorphic to the formal power series ring in $g+1+\delta$ variables with coefficients in $\ooo$ and define $\LL_\mathcal{W}:=\mathcal{W}[[\Gamma]]$.
\begin{define}
For a torsion $\LL$-module $\frak{X}$, define its characteristic ideal
$$\textup{char}(\frak{X})=\prod_{\frak{P}}\frak{P}^{\textup{length}(\frak{X}_\frak{P})},$$
where the product is over height-one prime of $\LL$.
\end{define}
Since $\LL$ is regular, note that the ideal $\textup{char}(\frak{X})$ is principal. Let $\chi:G_F\ra \ooo^\times$ be a Dirichlet character of order prime to $p$ and let $L_\chi$ be the finite extension of $F$ that $\chi$ cuts out. We define $M_\infty$ to be the maximal abelian pro-$p$ extension of $L_\infty=L_\chi F_\infty$ which is unramified outside $\Sigma_p$. Set $\frak{X}_\infty=\textup{Gal}(M_\infty/L_\infty)^\chi$, which is a finitely generated torsion $\LL$-module (see \cite[\S1.2]{ht94}). Then the main conjecture in this setting asserts that:
\begin{conj}[Main conjecture]
\label{conj:mainCMconjintro}
The characteristic ideal of $\frak{X}_\infty \otimes_{\ooo}\mathcal{W}$ is generated by the Katz $p$-adic $L$-function $\mathcal{L}_{\chi}^{\Sigma}\in \mathcal{W}[[\Gamma]]$.
 \end{conj}

When $K=\QQ$ and $F$ is a quadratic imaginary quadratic number field, Rubin in~\cite{rubinmainconj} proved this conjecture using the Euler system of elliptic units. %Rubin first bounds the characteristic ideal of   $\frak{X}_\infty$ in terms of elliptic units using the Euler system machinery, then he uses Yager's theorem ~\cite[Theorem 1]{yager} to relate his bound to Katz's two-variable $p$-adic $L$-function that was constructed in~\cite{katz76}.
In this paper, we use the conjectural Rubin-Stark elements (see \cite{ru96} for their conjectural description) in place of elliptic units. In order to do so, the general ``rank-$g$ Euler-Kolyvagin systems machinery'' (that was set up by the author in~\cite{kbbesrankr} refining Perrin-Riou's original approach in~\cite{pr-es}) needs to be appropriately modified to apply in this setting. After carrying this out, we prove the following theorem generalizing \cite[Theorems  9.3 and 10.6]{rubinmainconj}. For a finite extension $\mathcal{M}$ of $L_\chi$, let $\Sigma_M$ be the places of $\mathcal{M}$ that lie above those in $\Sigma$ and let $U_{\mathcal{M},\frak{p}}$ denote the $p$-adic completion of the local units of $\mathcal{M}$ at $\frak{p}$ and $U_{\mathcal{M}}=\oplus_{\frak{p}\in \Sigma_\mathcal{M}}U_{\mathcal{M},\frak{p}}$ be the semi-local units and $U_{\mathcal{M}}^\chi$ its $\chi$-isotypic part. Set $\frak{U}_\infty^\chi=\varprojlim U_{\mathcal{M}}^\chi$ and $\textup{loc}_{p}^{+,\otimes g}(\varepsilon_{F_\infty}^\chi) \in \wedge^g \frak{U}_\infty^\chi$ be the image of the tower of Rubin-Stark elements inside $\wedge^g \frak{U}_\infty^\chi$; see \S\ref{sec:RS} and \S\ref{subsec:IwTheory} for a precise definition of the element $\textup{loc}_{p}^{+,\otimes g}(\varepsilon_{F_\infty}^\chi)$.

\begin{thma}
\label{thm:mainconj1}\textup{(See Theorem~\ref{thm:mainconjforTchi} below)}. Assume that:
\begin{enumerate}
\item $\chi(\wp)\neq 1$ for any prime $\wp$ of $F$ above $p$ and $\chi\neq \omega$ (where $\omega$ is the $p$-adic Teichm\"uller character).
\item Hida and Tilouine's  $\Sigma$-Leopoldt conjecture holds true for $L_\chi$; see \cite[p. 94]{ht94} or Conjecture~\ref{conj:sigmaleopoldt} below  for its statement. %this implies the injectivity of the map $H^1(F,T_\chi\otimes\LL)\stackrel{\textup{loc}_p^+}{\lra} H^1_+(F_p,T_\chi\otimes\LL)$ as the kernel modulo the augmentation ideal is trivial, by Leopoldt's conjecture for L_\chi.
\item Rubin-Stark conjecture is true for all abelian extensions $\mathcal{K}/F$ chosen from the collection $\frak{E}_0$, which is defined as in \S\ref{subsec:notation} below.
\end{enumerate}
Then,
$$\textup{char}(\frak{X}_\infty)=\textup{char}(\wedge^g \frak{U}_\infty^\chi/\LL\cdot\textup{loc}_{p}^{+,\otimes g}(\varepsilon_{F_\infty}^\chi)).$$
\end{thma}
Here the exterior product is calculated in the category of $\LL$-modules. See Remark~\ref{thm:ESKSrestricted} which lends evidence for our assumption (3). See also Remark~\ref{rem:motivatesigmaleo} for a discussion regarding the $\Sigma$-Leopoldt conjecture.

The following conjecture we propose is a natural extension of Yager's theorem \cite[Theorem 1]{yager} to our setting:
\begin{conj}
\label{conj:yagerintro}
The principal ideal $\textup{char}\left(\wedge^g \frak{U}_\infty^\chi/\LL\cdot\textup{loc}_{p}^{+,\otimes g}(\varepsilon_{F_\infty}^\chi)\right)\LL_{\mathcal{W}}$ is generated by $\mathcal{L}_{\chi}^{\Sigma}$.
\end{conj}

In view of Theorem~A, this conjecture is equivalent to the main conjecture (namely, Conjecture~\ref{conj:mainCMconjintro} above). In particular, the recent work of Hsieh~\cite{hsiehCMmainconj} gives a partial result towards the truth of Conjecture~\ref{conj:yagerintro}:

\begin{thmb}
\label{thm:yagerwithhiseh}
Assume that the hypotheses of Theorem~A holds and suppose in addition that:
\begin{enumerate}
\item $p>5$ is prime to the minus part of the class number of $F$, to the order of $\chi$ and is unramified in $K/\QQ$.

\item $\chi$ is \emph{anticyclotomic} in the sense that $\chi(c\delta c^{-1})=\chi(\delta)^{-1}$ for $\delta\in \Delta$ and $c \in G_K$ that induces the generator of $\textup{Gal}(F/K)$.%G_F is normal in G_K so c\deltac^{-1} still lies in G_F.
\item $\chi$ is unramified at all places above $p$.

\item The restriction of $\chi$ to $G_{F(\sqrt{p^*})}$ (where $p^*=(-1)^{\frac{p-1}{2}}p$) is non-trivial.
\end{enumerate}
Then Conjecture~\ref{conj:yagerintro} is true.
\end{thmb}

As explained above, Theorem~B should be thought of as a generalization of an important theorem due to Yager~\cite{yager} to a general CM field $F$.

Assuming the truth of Conjecture~\ref{conj:yagerintro}, we obtain the following  result on the arithmetic of CM abelian varieties as a consequence of Theorem A; compare to \cite[Theorem 11.4]{rubinmainconj}. Let $A$ be an abelian variety defined over $F$ that has CM  by the ring of integers $\oo_F$ of $F$. For a fixed place $\varepsilon \in \Sigma$, let $\psi_\varepsilon$ be the associated (archimedean) Hecke character and let $\wp$ be a prime of $F$ above $p$ (whose choice depends in a precise way on the choice of the archimedean avatar $\psi_\varepsilon$ of the associated Gr\"ossencharacter, see~\S\ref{subsec:HeckecharCMab} below for precise definitions).%, let $A[\wp^\infty]$ denote the $\wp$-power torsion of $A$.%, and let $\textup{Sel}(F,A[\wp^\infty])$ the $\wp$-adic Selmer group.

\begin{thmc}[Theorem~\ref{thm:mainabvar}]
Assume that the hypotheses of Theorem A as well as Conjecture~\ref{conj:yagerintro} hold true. If $L(\psi_\varepsilon,0)$ is nonzero, then $A(F)$ is finite and $\Sha_{A/F}[\wp^\infty]$ is finite for sufficiently large primes $p$.
\end{thmc}
See \S\ref{subsec:HeckecharCMab} below for a detailed discussion; also \cite[Corollary 1]{hsiehCMmainconj} for a result along these lines on the arithmetic of CM elliptic curves, under different set of assumptions and proved using completely different techniques. %In a future work, we hope to prove this result without assuming the truth of Conjecture~\ref{conj:yagerintro}, and relying on Theorem A and the refinement of the Kolyvagin System machinery that is used to prove this theorem.

%The methods of the current paper are altogether different from the methods of Hsieh, and builds on our previous work~\cite{} on the Euler system machinery of rank $r$.

\subsection*{Acknowledgements} I would like to thank Ming-Lun Hsieh and Antonio Lei for their comments and suggestions on an earlier version of this article, and to Robert Pollack for his encouragement. I thank Jan Nekov\'a\v{r} for teaching me the general descent formalism built in his theory of Selmer complexes, which plays a crucial role in some of the proofs in this article. Special thanks are due to Karl Rubin, conversations with whom has essentially lead the author to write this article.
During the preparation of this article, I was partially supported by the European Commission (EC-FP7 Marie Curie grant 230668) and by T\"UB\.ITAK.
\subsection{Notation and Hypotheses}
\label{subsec:notation}$\,$

Given a group $G$, let $\pmb{\mu}(G)$ denote the torsion subgroup of $G$. For the local field $\frak{F}$ introduced above, assume that
\be\label{eqn:nopthrootsinF}p \hbox{ does not divide } |\pmb{\mu}(\frak{F}^\times)|.\ee

Let $\chi:G\ra \oo^\times$ be any continuous character. For an abelian group $A$ on which $G$ acts continuously, let $\hat{A}:\varprojlim_n A/A^{p^n}$ be the $p$-adic completion of $A$ and
$$A^{\chi}=\{a\in \hat{A}\otimes_{\ZZ_p}\oo: ga=\chi(g)a \hbox{ for all } g \in G\}.$$

For a  field $k$, fix a separable closure $\bar{k}$ of $k$ and let $G_k=\textup{Gal}(\bar{k}/k)$ be the absolute Galois group of $k$. When $k$ is a global field, $\mathbb{A}_k$ denotes the ad\`{e}le ring of $k$ and $\mathbb{A}^\times_k$ the group of id\'{e}les.

Denote by $F^{\textup{cyc}}$ the cyclotomic $\ZZ_p$-extension of $F$. Set $\Gamma^{\textup{cyc}}=\textup{Gal}(F^{\textup{cyc}}/F)$ and $\LL^{\textup{cyc}}=\frak{O}[[\Gamma^{\textup{cyc}}]]$. Let $\chi_{\textup{cyc}}: G_F \ra \ZZ_p^\times$ denote the cyclotomic character giving the action of $G_F$ on the $p$-power roots of unity $\pmb{\mu}_{p^{\infty}}$ and let $\langle \chi_{\textup{cyc}}\rangle:=\chi_{\textup{cyc}}\big{|}_{\Gamma^{\textup{cyc}}}: \Gamma^{\textup{cyc}}\ra \ZZ_p^\times$. Then $\omega=\langle \chi_{\textup{cyc}}\rangle^{-1} \chi_{\textup{cyc}}$ is the Teichm\"uller character, giving the action of $G_F$ on $\pmb{\mu}_{p}$.

 Suppose $\psi: G_F \twoheadrightarrow \frak{O}^\times$
is a continuous ($p$-adic) Hecke character with $\textup{im}(\psi)=\frak{O}^\times$ and which is Hodge-Tate. %Assume further that
%\be\label{eqn:assumepsiisnotcyclo}\psi \hbox{ is not the cyclotomic character } \chi_{\textup{cyc}}.\ee
Write $\frak{O}^\times=\pmb{\mu}(\frak{F}^\times)\times U^{(1)}$ and let $\langle\psi\rangle: G_F \twoheadrightarrow U^{(1)}$ be the map $\psi$ followed by the projection $\frak{O}^\times\twoheadrightarrow U^{(1)}$. Set
$$\omega_\psi=\psi\cdot\langle\psi\rangle^{-1}: G_F \twoheadrightarrow \pmb{\mu}(\frak{F}^\times) \hookrightarrow \frak{O}^{\times}$$
and assume that
\be
\label{eqn:chiisnotteich}
\omega_\psi \hbox{ is not the trivial character}.
\ee
%Define the $G_F$-representations
%$$T_0=\frak{O}(1)\otimes \omega_\psi,$$
%$$T=\frak{O}(\psi)=T_0\otimes \frak{O}(-1)\otimes\langle\psi\rangle,$$
%where $\frak{O}(1)=\frak{O}\otimes_{\ZZ_p}\ZZ_p(1)$ and
Let $\psi^*=\psi^{-1}\otimes\chi_{\textup{cyc}}$ and $T=\frak{O}(\psi^*)$ be the free $\frak{O}$ module of rank one on which $G_F$ acts via the Hecke character $\psi^*$. %Let $L$ be the finite extension of $F$ cut by the character $\omega_\psi$. Set $\Delta=\textup{Gal}(L/F)$; note that $p\nmid |\Delta|$.
Given a Hecke character $\psi$ as above, let $F_\psi/F$ be the $\ZZ_p^{[\frak{F}:\QQ_p]}$-extension that $\langle \psi\rangle$ factors through. Write $\Gamma_\psi=\textup{Gal}(F_\psi/F)$ and $\LL_\psi=\ZZ_p[[\Gamma_\psi]]$.

 Let $\chi:G_F \ra \ooo^\times$ be any Dirichlet character (whose order is necessarily prime to $p$) which has the property that
\be
\label{eqn:assnotrivxhi}
\chi(\wp)\neq 1 %\neq \chi\omega^{-1}(\wp)
\hbox{  \, for any prime   } \wp \hbox{ of } F \hbox{ above } p.
\ee
and that
\be
\label{eqn:chiisnotteich1}
\chi \neq \omega. %\neq \chi\omega^{-1}(\wp)
\ee

Define the $G_F$-representation $T_\chi=\frak{O}(1)\otimes \chi^{-1}$.  %Observe that $T_\chi = T \otimes  \langle\psi\rangle^{-1} \langle \chi_{\textup{cyc}}\rangle$.
Let $L_\chi$ be the finite extension of $F$ cut by the character $\chi$ and set $\Delta=\textup{Gal}(L_\chi/F)$; note that $p\nmid |\Delta|$. Although for applications of our results to a CM abelian variety $A$, we will only need to study the case $\chi=\omega_\psi$ where $\psi$ is the $p$-adic Hecke character attached to $A$, we still chose to state some of our results in greater generality.

Let $\mathcal{R}$ be the set of primes of $F$ that does not contain any prime above $p$ nor any prime at which $\chi$ is ramified. Define $\NN(\mathcal{R})$ to be the square free products of primes chosen from $\mathcal{R}$. For $\ell\in \mathcal{R}$, let $F(\ell)$ be the maximal $p$-extension inside the ray class field of $F$ modulo $\ell$ and for $\eta=\ell_1\cdots\ell_s \in \NN(\mathcal{R})$, set $F(\eta)=F(\ell_1)\cdots F(\ell_s)$. We write $L(\eta)=L\cdot F(\eta)$ for the composite field. We define the collections of finite abelian extensions of $F$ (resp., of $L$)
$$\frak{E}=\{M\cdot F(\eta): \eta \in \NN(\mathcal{R}); M\subset F_\infty \hbox{ is a finite extension of } F\},$$
$$\frak{E}_0=\{M\cdot L(\eta): \eta \in \NN(\mathcal{R}); M\subset F_\infty \hbox{ is a finite extension of } F\},$$
%$$\tilde{\frak{E}}=\{N(\pmb{\mu}_p): N \in \frak{E}_0\},$$
%$$\tilde{\frak{E}}_0=\{N(\pmb{\mu}_p): N \in \frak{E}\}.$$
Let $\displaystyle{\frak{K}_0=\varinjlim_{N \in \frak{E}_0}N}$ and $\displaystyle{\frak{K}=\varinjlim_{N \in \frak{E}}N}$. %, $\displaystyle{\tilde{\frak{K}}_0=\varinjlim_{N \in \tilde{\frak{E}}_0}N}$, and $\displaystyle{\tilde{\frak{K}}=\varinjlim_{N \in \tilde{\frak{E}}}N}$.
Set $\frak{G(X)}=\textup{Gal}(\frak{X}/F)$ and $\LL_{\frak{X}}=\frak{O}[[\frak{G(X)}]]$ for $\frak{X}=\frak{K}_0$ and ${\frak{K}}$.%, \tilde{\frak{K}}_0$, and $\tilde{\frak{K}}$.

%The following hypotheses will play a role in what follows:

%\begin{itemize}
%\item[($\mathbf{{H}.nE}$)] $\psi^{-1}\otimes \chi_{\textup{cyc}}\big{|}_{F_\wp^\times}$ is not the trivial character for any prime $\wp$ of $F$ above $p$.
%\end{itemize}

For any non-archimedean prime $\lambda$ of $F$, fix a decomposition group $\mathcal{D}_{\lambda}$ and the inertia subgroup $\mathcal{I}_\lambda \subset \mathcal{D}_{\lambda}$. Let ${(-)}^\vee=\textup{Hom}(-,\QQ_p/\ZZ_p)$ denote Pontryagin duality functor.
%\subsection{Statements of the main results}
%\newpage
\section{Selmer structures and Selmer groups}
\label{sec:selmer}
\subsection{Semi-local Preparation}
Let $M=M_0\cdot F(\eta)$ be a member of the collection $\frak{E}$, where $M_0$ is a finite subextension of $F_\infty/F$. Set $\Delta_M=\textup{Gal}(M/F)$, $\delta_M=|\Delta_M|$ and $\LL_M=\frak{O}[\Delta_M]$.

%Throughout, the hypothesis $\mathbf{{H}.nE}$ is in effect.
Let $X$ be any $\frak{O}[[G_F]]$-module which is free of rank $d$ as an $\ooo$-module. Suppose in addition that $X$ satisfies the following hypothesis:
\begin{enumerate}
\item[\textbf{(H.p1)}] $H^2(F_\wp,X)=0=H^2\left(F_\wp,\textup{Hom}_{\ooo}(X,\ooo(1))\right)$, for any prime $\wp$ of $F$ above $p$.
\end{enumerate}
%\begin{prop}
%\label{prop:localrank}
%The $\LL^{\textup{cyc}}$-modules
%$$H^1_{+}(M_{p},X\otimes\LL^{\textup{cyc}}):=\bigoplus_{i=1}^s \bigoplus_{\frak{q}|\wp_i}H^1(M_{\frak{q}},X\otimes\LL^{\textup{cyc}})$$
%and
%$$H^1_{-}(M_{p},X\otimes\LL^{\textup{cyc}}):=\bigoplus_{i=1}^s \bigoplus_{\frak{q}|\wp_i^c}H^1(M_{\frak{q}},X\otimes\LL^{\textup{cyc}})$$
%are both free of rank $g\cdot[M:F]$.
%\end{prop}

\begin{lemma}
\label{lem:extendp1toM}
Suppose $X$ is above. Let $M \in \frak{E}$ be an extension of $F$ and let $\frak{P}$ be a prime of $M$ lying above $p$. Then
 $$H^2(M_\frak{P},X)=0=H^2\left(M_\frak{P},\textup{Hom}_{\ooo}(X,\ooo(1))\right).$$
\end{lemma}

\begin{proof}
Let $\wp$ be the prime of $F$ lying below $\frak{P}$ and set $D_\frak{P}=\textup{Gal}(M_\frak{P}/F_\wp)$. Then either $D_\frak{P}$ is trivial and in this case  Lemma follows from \textup{\textbf{(H.p1)}}, or otherwise $D_\frak{P}$ is a non-trivial $p$-group. Then,
$$\# H^0(M_{\frak{P}},X^*[\varpi])= \# H^0\left(D_{\frak{P}}, (H^0(M_{\frak{P}},X^*[\varpi])\right)\equiv \#H^0(F_\wp,X^*[\varpi]) \equiv 1 \mod p$$
where the last equality holds thanks to \textup{\textbf{(H.p1)}} and local duality. This shows that $ H^0(M_{\frak{P}},X^*)=0$ and thus by local duality that $H^2(M_{\frak{P}},X)=0$, as desired. The second assertion is proved in an identical manner.
\end{proof}
\begin{define}
\label{def:pmsemilocal}
For $j=0,1,2$ define
$$H^j_{+}(M_{p},X):=\bigoplus_{i=1}^s \bigoplus_{\frak{q}|\wp_i}H^j(M_{\frak{q}},X)$$
and
$$H^j_{-}(M_{p},X):=\bigoplus_{i=1}^s \bigoplus_{\frak{q}|\wp_i^c}H^j(M_{\frak{q}},X).$$
\end{define}
\begin{prop}
\label{prop:localrank}
For $M$ and $X$ as above, the $\ooo$-modules $H^1_{+}(M_{p},X)$ and $H^1_{-}(M_{p},X)$
are both free of rank $g\cdot d\cdot \delta_M$.
\end{prop}

\begin{proof}
 Let $\frak{P}$ be any prime of $M$ above $p$. By \cite[Prop. 4.2.9]{nekovar06}, the cohomology
$H^{\bullet}(M_\frak{P},X)$ is represented by a perfect complex of
$\ooo$-modules (i.e., projective (hence free) $\ooo$-modules of finite
type) concentrated in degrees $0,1$ and $2$. In particular, since we
assume that $H^{2}(M_\frak{P},X)=0$, then this complex may be taken in
degrees $0$ and $1$. Similarly, the local cohomology
$H^{\bullet}\left(M_\frak{P},\textup{Hom}_{\ooo}(X,\ooo(1))\right)$ is represented by a
perfect complex of $\ooo$-modules concentrated in degrees $0$ and $1$.
The two complexes
$H^{\bullet}(M_\frak{P},X)$ and
$H^{\bullet}\left(M_\frak{P},\textup{Hom}_{\ooo}(X,\ooo(1))\right)$ are related by the
duality functor $\textup{RHom}_{\ooo}(-,\ooo)[-2]$ (c.f., \cite[Prop.
5.2.4]{nekovar06}). As a result, each of these two complexes is also
represented by a perfect complex concentrated in degrees $2-1=1$ and
$2-0=2$, hence by a single projective (hence free) $\ooo$-module of
finite type in degree $1$. This shows that both $H^1_{+}(M_{p},X)$ and $H^1_{-}(M_{p},X)$ are free $\oo$-modules of finite type.

Let $M^{\textup{cyc}}/M$ be the cyclotomic $\ZZ_p$-extension and let $\LL^{\textup{cyc}}_M=\ooo[[\Gamma_M^{\textup{cyc}}]]$ be the completed group ring of its Galois group $\Gamma_M^{\textup{cyc}}=\textup{Gal}(M^{\textup{cyc}}/M)$.  Let $\gamma_M$ be a topological generator of $\Gamma_M^{\textup{cyc}}$. Since the cohomological dimension of $G_{M_\frak{P}}$ is 2 and since we assumed \textbf{(H.p1)}, it follows that
$$H^2(M_\frak{P},X\otimes\LL_M^{\textup{cyc}})/(\gamma_M-1)\cong H^2(M_\frak{P},X)=0,$$
and hence by Nakayama's lemma that
\be\label{eqn:h2vanishforlamdacyc}H^2(M_\frak{P},X\otimes\LL_M^{\textup{cyc}})=0.\ee
Furthermore, as the ring $\LL_M^{\textup{cyc}}$ is Gorenstein, Nekov\'a\v{r}'s machinery applies as above \emph{verbatim} for the $\LL_M^{\textup{cyc}}$-module $X\otimes\LL_M^{\textup{cyc}}$ to conclude that the $\LL^{\textup{cyc}}_M$-module $H^1(M_{\frak{P}},X\otimes\LL^{\textup{cyc}}_M)$ is free of finite rank.  Its rank equals $[M_\frak{P}:\QQ_p]\cdot d$ by \cite[Theorem A.8(ii)]{kbbiwasawa}. The natural map
$$H^1(M_\frak{P},X\otimes\LL_M^{\textup{cyc}}) \lra H^1(M_\frak{P},X)$$
is surjective by (\ref{eqn:h2vanishforlamdacyc}). We conclude therefore that $ H^1(M_\frak{P},X)$ is a free $\ooo$-module of rank $[M_{\frak{P}}:\QQ_p]\cdot d$. This completes the proof.
\end{proof}
\begin{lemma}
\label{lem:corestrictionsurjective}
If \textup{\textbf{(H.p1)}} holds true then the corestriction map
$$\textup{cor}: H^1_\pm(M_p,X)\lra H^1_\pm(F_p,X)$$
is surjective.
\end{lemma}
\begin{proof}
Define $X_M=\textup{Ind}_{M/F} X$ and let $\frak{A}_M$ be the augmentation ideal of the local ring $\ooo[\Delta_M]$. By the semi-local Shapiro's Lemma \cite[\S A.5]{r00}, there is a canonical isomorphism
$$H^1_\pm(M_p,X) \cong H^1_\pm(F_p,X_M)$$
and the map
$$\textup{cor}_{M/F}: H^1_\pm(F_p,X_M)\cong H^1_\pm(M_p,X) \lra H^1_\pm(F_p,X)$$
is induced from the augmentation sequence
$$0\lra \frak{A}_M\cdot X_M \lra X_M \lra X \lra 0.$$
The cokernel of $\textup{cor}_{M/F}$ is therefore contained in
$H^2_\pm(F_p,\frak{A}_M\cdot X_M)$. To conclude the proof, it therefore suffices to check that $H^2(F_\wp,\frak{A}_M\cdot X_M)=0$ for every prime $\wp$ of $F$ lying above $p$. This is equivalent by local duality to checking that $H^0(F_\wp,(\frak{A}_M\cdot X_M)^*)=0$. By the Claim on Page 1303 of \cite{kbbesrankr}, it follows that $H^0(F_\wp,(\frak{A}_M\cdot X_M)^*)\hookrightarrow H^0(F_\wp,X_M^*)$, hence the proof reduces to verify that $H^0(F_\wp,X_M^*)=0$, which again by local duality is equivalent to the vanishing of $H^2(F_\wp,X_M)\cong H^2(M_\frak{P},X)$, and this is the conclusion of Lemma~\ref{lem:extendp1toM}.

\end{proof}
\begin{prop}
\label{prop:localrank}
For $M$ and $X$ as above then the $\LL_M$-modules
$H^1_{\pm}(M_{p},X)$ are both free of rank $g\cdot d$.
\end{prop}
\begin{proof}
By Lemma~\ref{lem:corestrictionsurjective}, the map $H^1_\pm(M_p,X)\ra H^1_\pm(F_p,X)$ (which may be thought of as reduction modulo the augmentation ideal $\frak{A}_M \subset \ooo[\Delta_M]$) is surjective. Nakayama's Lemma and Lemma~\ref{prop:localrank} therefore imply that $H^1_\pm(M_p,X)$ is generated by (at most) $g\cdot d$ elements over the ring $\ooo[\Delta_M]$. Let $\frak{B}=\{x_1, x_2,\dots,x_{g\cdot d}\}$ be any set of such generators. To prove (i), it suffices to check that the $x_i$'s do not admit any non-trivial $\ooo[\Delta_M]$-linear relation. Assume contrary, and suppose there is a non-trivial relation
 \be\label{eqn:relation}
 \sum_{i=1}^{g\cdot d}\alpha_i x_i=0, \,\,\, \alpha_i \in \ooo[\Delta_M].
 \ee
 Write $S=\{\delta x_j: \delta \in \Delta_M, 1\leq j\leq g\cdot d\},$ note that by our assumption on the set  $\frak{B}$, the set $S$ generates $H^1_\pm(M_p,X)$ as an $\ooo$-module, and $|S|=g\cdot d\cdot\delta_M=\textup{rank}_{\ooo} \, H^1_\pm(M_p,X)$.  Equation (\ref{eqn:relation}) may be rewritten as
 $$\sum_{\delta,j} a_{\delta,j}\cdot \delta x_j=0$$
  with $a_{\delta,j} \in \ooo$. Since we already know that $H^1(M_p,X)$ is $\ooo$-torsion free, we may assume without loss of generality that $a_{\delta_0,j_0} \in \ooo^\times$ for some $\delta_0,j_0$. This in turn implies that
 $$\delta_0x_{j_0} \in \textup{span}_{\ooo}(S-\{\delta_0x_{j_0}\}),$$
  hence $H^1(M_p,X)$ is generated by $S-\{\delta_0x_{j_0}\}$. This, however, is a contradiction since we already know that the $\ooo$-rank of $H^1(M_p,T)$ is $g\cdot d\cdot\delta_M=|S|$, hence it cannot be generated by $|S|-1$ elements over $\ooo$. The proof of the Proposition is now complete.
\end{proof}

\begin{cor}
\label{cor:localfull}$\,$
\begin{itemize}
\item[(i)] The $\LL$-modules $H^1_{\pm}(F_p,X\otimes\LL)$ are both free of rank $g\cdot d$.
\item[(ii)] The $\LL_{\frak{K}}$-modules
$\displaystyle{\varprojlim_{M\in \frak{E}}H^1_{\pm}(M_{p},X)}$, where the inverse limits are with respect to corestriction maps,
are both free of rank $g\cdot d$.
\end{itemize}
\end{cor}
\begin{proof}
Immediate after Proposition~\ref{prop:localrank}.
\end{proof}
When $\chi=\omega_\psi$, observe that $T=T_\chi\otimes \langle \psi\rangle^{-1}$. The hypothesis $\textbf{(H.p1)}$ is verified for $X=T$ and  $X=T_\chi$, since we assumed %(\ref{eqn:assumepsiisnotcyclo}),
(\ref{eqn:chiisnotteich}) and (\ref{eqn:assnotrivxhi}). In particular, the conclusions of Corollary~\ref{cor:localfull} hold true both choices of $G_F$-representations.
\begin{define}
\label{def:modifiedlocalcondition}
\begin{enumerate}
\item Let $\frak{L}$ be any free rank one $\LL_{\frak{K}}$-direct summand of $\displaystyle{\varprojlim_{M\in \frak{E}}H^1_{+}(M_{p},T_\chi)}$.
%\item[(1)] For $\eta\in\NN(\RR)$, let $\frak{L}_\eta \subset H^1(F(\eta)_p,T\otimes\LL)$ be a free rank one $\LL_\eta$-direct summand of $H^1_{+}(F(\eta)_{p},T\otimes\LL)$. When $F(\eta)=F$, we simply write $\frak{L}$ for this rank one direct summand.
\item For $M=M_0\cdot F(\eta) \in \frak{E}$ with $F\subset M_0\subset F_\infty$,  let $\al_M \subset H^1_{+}(M_{p},T_\chi)$ be the image of $\frak{L}$ under the surjection
$$\varprojlim_{N}H^1(N_p,T_\chi)\ra H^1(M_{p},T_\chi).$$
We write $\al$ instead of $\al_F$.
\item Let $\al_\infty$ be the image of $\frak{L}$ under the surjection
$$\varprojlim_{N}H^1(N_p,T_\chi)\ra H^1(F_{p},T_\chi\otimes\LL).$$
\end{enumerate}
\end{define}
\begin{define}
\begin{itemize}
\item[(1)] The submodule
$$H^1_{\FF_{\al_\infty}}(F_p,T_\chi\otimes\LL)=H^1_{-}(F_p,T_\chi\otimes\LL)\oplus \al_\infty\subset H^1(F_p,T_\chi\otimes\LL)$$
is called the \emph{$\al_\infty$-modified local condition} on $T_\chi\otimes\LL$.
\item[(2)] Similarly, the submodule
$$H^1_{\FF_{\al}}(F_p,T_\chi)=H^1_{-}(F_p,T_\chi)\oplus \al\subset H^1(F_p,T_\chi)$$
is called the \emph{$\al$-modified local condition} on $T$.
\end{itemize}
\end{define}

\subsection{Selmer structures}
\label{subsec:selmerstr}
The notation that we have set above is in effect.

We first recall Mazur and Rubin's definition of a \emph{Selmer structure}, in particular the \emph{canonical Selmer structure} on $T_\chi$ and $T_\chi\otimes\LL$ (or on their various twists).

 Let $R$ be a complete local noetherian $\ooo$-algebra, and let $X$ be a $R[[G_F]]$-module which is free of finite rank over $R$. In this paper, we will be interested in the case when $R=\LL$ or its certain quotients, and $X$ is $T\otimes\LL$ or its relevant quotients by an ideal of $\LL$. (For example, taking the quotient by the augmentation ideal of $\LL$ will give us $\ooo$ and the representation $T$.)

 \begin{define}
\label{selmer structure}
A \emph{Selmer structure} $\FF$ on $X$ is a collection of the following data:
\begin{itemize}
\item a finite set $\Sigma(\FF)$ of places of $F$, including all infinite places and primes above $p$, and all primes where $X$ is ramified.
\item for every $\lambda \in \Sigma(\FF)$ a local condition on $X$ (which we view now as a $R[[\mathcal{D}_{\lambda}]]$-module), i.e., a choice of $R$-submodule
$$H^1_{\FF}(F_{\lambda}, X) \subset H^1(F_{\lambda}, X).$$
 \end{itemize}
If $\lambda \notin \Sigma(\FF)$ we will also write
$H^1_{\FF}(F_{\lambda}, X)=H^1_{f}(F_{\lambda}, X)$, where the module
$H^1_{{f}}(F_{\lambda}, X)$ is the \emph{finite} part of
$H^1(F_{\lambda}, X)$, defined as in~\cite[Definition 1.1.6]{mr02}.
\end{define}

\begin{define}
 The \emph{semi-local cohomology group} at a rational
prime $\ell$ is defined by setting
$$H^i(F_\ell, X):=\bigoplus_{\lambda|\ell} H^i(F_\lambda, X).$$
\end{define}
 Let $\lambda$ be a prime of $F$. There is the perfect local
Tate pairing
$$<\,,\,>_\lambda\,:H^1(F_\lambda,X) \times
H^1(F_\lambda,X^*) \lra H^2(F_\lambda,\frak{F}/\ooo(1))
\stackrel{\sim}{\lra}\frak{F}/\ooo,$$
where we recall that $X^*:=\textup{Hom}(X, \pmb{\mu}_{p^{\infty}})$ is the Cartier dual of $X$. For a Selmer structure $\FF$ on $X$, define $H^1_{\FF^*}(F_\lambda,X^*):=H^1_\FF(F_\lambda,X)^\perp$ as
the orthogonal complement of $H^1_\FF(F_\lambda,X)$ with respect to
the local Tate pairing. The Selmer structure $\FF^*$ on $X^*$ (with
$\Sigma(\FF)=\Sigma(\FF^*)$) defined in this way will be called the
\emph{dual Selmer structure}.

For examples of local conditions see~\cite[Definitions 1.1.6 and 3.2.1]{mr02}.
\begin{define}
\label{selmer group}
If $\FF$ is a Selmer structure on $X$, we define the \emph{Selmer module} $H^1_{\FF}(F,X)$ as
 $$H^1_{\FF}(F,X):=\ker\left(H^1(\textup{Gal}(F_{\Sigma(\FF)}/F),X) \lra \bigoplus_{\lambda \in \Sigma(\FF)}H^1(F_{\lambda},X)/H^1_{\FF}(F_{\lambda},X)\right),$$
 where $F_{\Sigma(\FF)}$ is the maximal extension of $F$ which is unramified outside $\Sigma(\FF)$. We also define the dual Selmer structure in a similar way; just replace $X$ by $X^*$ and $\FF$ by $\FF^*$ above.
\end{define}

\begin{example}
\label{example:canonical selmer}
In this example we recall~\cite[Definitions 3.2.1 and  5.3.2]{mr02} of which we make frequent use.
\begin{itemize}
\item[(i)] Let $R=\ooo$ and let ${X}$ be a free $R$-module endowed with a continuous action of $G_F$, which is unramified outside a finite set of places of $F$.  We define a Selmer structure $\FFc$ on ${X}$ by setting $\Sigma(\FFc)=\{\lambda: X \hbox{ is ramified at } \lambda\}\cup\{\wp|p\}\cup\{v|\infty\}$, and
\begin{itemize}
\item if $\lambda \in \Sigma(\FFc)$, $\lambda\nmid p\infty$, we define the local condition at $\lambda$ to be
$$H^1_{\FFc}(F_\lambda, {X})=\ker(H^1(F_\lambda, {X}) \lra H^1(F_\lambda^{\textup{unr}},X\otimes \frak{F})),$$
where $F_\lambda^{\textup{unr}}$ is the maximal unramified extension of $F_\lambda$,
\item if $\wp|p$, we define the local condition at $\wp$ to be
$$H^1_{\FFc}(F_\wp, {X})=H^1(F_\wp,X).$$
\end{itemize}
The Selmer structure $\FFc$ is called the \emph{canonical Selmer structure} on ${X}$.
\item[(ii)] Let now $R$ be the Iwasawa algebra $\LL$ or its cyclotomic quotient $\LL^{\textup{cyc}}$, and let $\mathbb{X}$ be a free $R$-module endowed with a continuous action of $G_F$, which is unramified outside a finite set of places of $F$.  We define a Selmer structure $\FF_R$ on $\mathbb {X}$ by setting
$$\Sigma(\FF_R)=\{\lambda: \mathbb{X} \hbox{ is ramified at } \lambda\}\cup\{\wp\subset F: \wp|p\}\cup\{v|\infty\},$$
 and $H^1_{\FF_R}(F_\lambda, \mathbb {X})=H^1(F_\lambda, \mathbb{X})$ for every $\lambda \in \Sigma(\FF_R)$. The Selmer structure $\FF_{R}$ is called the \emph{canonical  $R$-adic Selmer structure} on $\mathbb{X}$.
 \end{itemize}
We still denote the induced Selmer structure on the quotients
$\mathbb{X}/I\mathbb{X}$ by $\FF_R$, which is obtained by
\emph{propagating} $\FF_R$ on $\mathbb{X}$ (see~\cite[Example
1.1.2]{mr02}). Note for $\lambda\in \Sigma(\FF_{R})$ that
$H^1_{\FF_R}(F_{\lambda}, \mathbb{X}/I\mathbb{X})$ will not always
be the same as $H^1(F_{\lambda}, \mathbb{X}/I\mathbb{X})$. In
particular, when $I$ is the augmentation ideal of $R$, the
Selmer structure $\FF_R$ on $\mathbb{X}$ will not always propagate
to $\FFc$ on $X:=\mathbb{X}\otimes_{R}R/I.$ However, when $X=T$
and $\mathbb{X}=T\otimes_{\ooo}R$ as in~\S\ref{sec:intro}, $\FF_R$
on $\mathbb{X}$ \emph{does} propagate  to $\FFc$ on $X$, under the
hypothesis $\hne$, as we shall check below.

\end{example}
%\begin{rem}
%\label{rem:canonical selmer} We say that an element $f\in \LL$ is
%distinguished if $\LL/(f)$ is a free $\oo$-module of finite rank.
%When $R=\LL$ and $\ \mathbb{T}=T\otimes_{\oo}\LL$ (which is one of
%the cases of interest), the Selmer structure $\FFc$ defined
%in~\cite[ \S2.1]{kbb} on the quotients $T\otimes_{\oo}\LL/(f)$ may
%be identified, under the hypotheses $\htam$ and $\hne$, by the
%propagation of $\FF_\LL$  to the quotients $T\otimes_{\oo}\LL/(f)$,
%for every distinguished $f\in \LL$. Indeed, for every prime
%$\lambda\subset k$, the submodule
%$$H^1_{\FFc}(F_\lambda,T\otimes_{\oo}\LL/(f))\subset
%H^1(F_\lambda,T\otimes_{\oo}\LL/(f))$$ is the image of the canonical
%map $H^1(F_\lambda,T\otimes_{\oo}\LL) \ra
%H^1(F_\lambda,T\otimes_{\oo}\LL/(f))$, by the proofs
%of~\cite[Proposition 2.10 and 2.12]{kbb}. By definition,
%$H^1_{\FF_\LL}(F_\lambda,T\otimes_{\oo}\LL/(f))$ is exactly the same
%thing.
%\end{rem}

\begin{define}
\label{def:selmer triple}
A \emph{Selmer triple} is a triple $(X,\FF,\PP)$, where $\FF$ is a Selmer structure on $X$ and $\PP$ is a set non-archimedean primes of $F$ disjoint from $\Sigma(\FF)$.
\end{define}

\begin{define}
 \label{def:line over k_infty}
 \begin{itemize}
 \item[(a)] Let $\FF_-$ be the Selmer structure on $T_\chi\otimes\LL$ defined as follows:
 \begin{itemize}
 \item $\Sigma(\FF_-)=\Sigma(\FF_{\LL})$,
 \item if $\lambda \nmid p$, define $H^1_{\FF_-}(F_\lambda, T_\chi\otimes\LL)=H^1_{\FF_\LL}(F_\lambda,T_\chi\otimes\LL)$,
 \item $H^1_{\FF_-}(F_p,T_\chi\otimes\LL):=H^1_{-}(F_p,T_\chi\otimes\LL) \subset H^1(F_p,T_\chi\otimes\LL)=H^1_{\FF_\LL}(F_p,T_\chi\otimes\LL)$.
 \end{itemize}

\item[(b)] Fix a $\LL$-rank one direct summand $\al_\infty \subset H^1_{+}(F_p,T_\chi\otimes\LL)$ as in Definition~\ref{def:modifiedlocalcondition}. Define the \emph{$\al_\infty$-modified Selmer structure} $\FF_{\al_\infty}$ on $T_\chi\otimes\LL$ as follows:
 \begin{itemize}
 \item $\Sigma(\FF_{\al_\infty})=\Sigma(\FF_{\LL})$,
 \item if $\lambda \nmid p$, define $H^1_{\FF_{\al_\infty}}(F_\lambda, T_\chi\otimes\LL)=H^1_{\FF_\LL}(F_\lambda,T_\chi\otimes\LL)$,
 \item $H^1_{\FF_{\al_\infty}}(F_p,T_\chi\otimes\LL):=H^1_{-}(F_p,T_\chi\otimes\LL)\oplus \al_\infty \subset H^1_{\FF_\LL}(F_p,T_\chi\otimes\LL)$.
 \end{itemize}
 \end{itemize}
 \end{define}

\subsection{Comparing Selmer groups}
\label{subsec:compareselmer}
To ease notation, set $\TT=T_\chi\otimes\LL$.  Suppose in this section that $\chi=\omega_\psi$.
\begin{lemma}
\label{lem:twistedisom} We have the following isomorphisms of $\LL$-modules:
\be\label{eqn:globaltwisted}H^1(F,\TT) \otimes\langle\psi\rangle^{-1}\cong H^1(F,T\otimes\LL),\ee
\be\label{eqn:localtwisted}H^1(F_\ell,\TT) \otimes\langle\psi\rangle^{-1}\cong H^1(F_\ell,T\otimes\LL)\ee
for every prime $\ell$.
\end{lemma}
\begin{proof}
As we have remarked above, $T_\chi=T\otimes \langle\psi\rangle^{-1}$. As $\langle\psi\rangle^{-1}$ is a continuous character of $\Gamma$, the proof of the Lemma follows from \cite[Proposition 6.2.1]{r00}.

\end{proof}
\begin{define}
\label{def:FminusonTchi}	
Let $H^1_{\FF_-}(F_p,T\otimes\LL) \subset H^1(F_p,T\otimes\LL)$ (resp., $\al_\infty^\psi$) be the isomorphic image of $H^1_{\FF_-}(F_p,\TT)$ (resp., $\al_\infty$) under the isomorphism (\ref{eqn:localtwisted}) above. For any subquotient $X$ of $T\otimes\LL$, let $H^1_{\FF_-}(F_\ell,X)\subset H^1(F_\ell,X)$ denote the propagated local condition on $X$, in the sense of \cite[Example 1.1.2]{mr02}.
\end{define}

\begin{lemma}
\label{lem:twistedglobalisom} We have the following isomorphism of Selmer groups:
$$H^1_{\FF_-}(F,\TT) \otimes\langle\psi\rangle^{-1} \cong H^1_{\FF_-}(F,T\otimes\LL).$$
\end{lemma}
\begin{proof}
Immediate from Lemma~\ref{lem:twistedisom} and the definition of the local condition $\FF_{-}$ on $\TT$ and on its twist $T\otimes\LL$.
\end{proof}
\begin{lemma}
\label{lem:structure for units}
The $\ooo$-module $\oo_{L_\chi}^{\times,\chi}$ is free of rank $g$.
\end{lemma}
\begin{proof}
This follows from \cite[\S8.6.12]{neukirch}, along with our assumption that $\chi$ is different from the Teichm\"uller character $\omega$.
\end{proof}
Consider the following hypothesis:
\be\label{eqn:LCforcedvanishing}
H^1_{\FF_-^*}(F,T_\chi^*) \hbox{ is finite.}
\ee
We will verify later (see Theorem~\ref{thm:mainconjforTchi} below) that (\ref{eqn:LCforcedvanishing}) holds true if we assume $\Sigma$-Leopoldt conjecture of Hida and Tilouine.
\begin{lemma}
\label{lem:minusvanishingforunits}
Assuming \textup{(\ref{eqn:LCforcedvanishing})}, $H^1_{\FF_-}(F,T_\chi)=0$.
\end{lemma}
\begin{proof}
Let $\FFc$ denote the canonical Selmer structure on $T_\chi$, defined as in Example~\ref{example:canonical selmer}. As explained in \cite[Lemma 6.1.2]{mr02} and \cite[Proposition II.2.6]{r00}, it follows from our assumption~(\ref{eqn:assnotrivxhi}) that $H^1_{\FFc}(F,T_\chi)\cong \oo_{L_\chi}^{\times,\chi}$, and therefore $H^1_{\FFc}(F,T_\chi)$ is an $\frak{O}$-module of rank $g$ by Lemma~\ref{lem:structure for units} under the running assumptions.

The $\LL$-module $H^1(F_p,T_\chi\otimes\LL)$ is free of rank $2g$ by Lemma~\ref{lem:twistedisom} and Corollary~\ref{cor:localfull}, and $H^1_{-}(F_p,T_\chi\otimes\LL)$ is a free rank-$g$ direct summand of $H^1(F_p,T_\chi\otimes\LL)$. Furthermore, the natural map
$$H^1(F_p,T_\chi\otimes\LL) \lra H^1(F_p,T_\chi)$$
is surjective thanks to the assumption (\ref{eqn:assnotrivxhi}). It thus follows that $H^1(F_p,T_\chi)$ is a free $\frak{O}$-module of rank $2g$ and $H^1_{\FF_{-}}(F_p,T_\chi)$ is a  direct summand of this module of rank $g$. We conclude that
$$\textup{rank}_{\frak{O}}\, H^1_{\FFc}(F_p,T_\chi) -\textup{rank}_{\frak{O}}\,H^1_{\FF_{-}}(F_p,T_\chi)=g.$$
Finally, it follows from \cite[Proposition 1.6]{wiles} that
\begin{align*}
\left(\textup{rank}_{\frak{O}}\, H^1_{\FFc}(F,T_\chi)-\textup{corank}_{\frak{O}}\, H^1_{\FFc^*}(F,T_\chi^*)\right)-&\left(\textup{rank}_{\frak{O}}\, H^1_{\FF_{-}}(F,T_\chi)-\textup{corank}_{\frak{O}}\, H^1_{\FF_{-}^*}(F,T_\chi^*)\right)\\
&=\textup{rank}_{\frak{O}}\, H^1_{\FFc}(F_p,T_\chi) -\textup{rank}_{\frak{O}}\,H^1_{\FF_{-}}(F_p,T_\chi)\\
&=g.
\end{align*}
Since $ H^1_{\FFc^*}(F,T_\chi^*)$ is finite and $\textup{rank}_{\frak{O}}\, H^1_{\FFc}(F,T_\chi)=g$, we conclude that
$$\textup{rank}_{\frak{O}}\, H^1_{\FF_{-}}(F,T_\chi)=\textup{corank}_{\frak{O}}\, H^1_{\FF_{-}^*}(F,T_\chi^*).$$
Since we assumed (\ref{eqn:LCforcedvanishing}), the proof follows.
\end{proof}

\begin{prop}
\label{prop:minusselmervanishes}
Assuming \textup{(\ref{eqn:LCforcedvanishing})}, $H^1_{\FF_-}(F,T\otimes\LL)=H^1_{\FF_-}(F,\TT)=0$.
\end{prop}

\begin{proof}
%It  suffices to check by Lemma~\ref{lem:twistedglobalisom} that $H^1_{\FF_-}(F,T_\chi\otimes\LL)=0$.
Let $\Gamma\cong \Gamma_1 \times \cdots  \times\Gamma_{g+1-\delta},$
where $\Gamma_i\cong\ZZ_p$ and $\delta$ is Leopoldt's defect.  Let $\gamma_i$ be a topological generator of $\Gamma_i$. %and let $\LL_i=\frak{O}[[\Gamma_i]]$. The exact sequence $$0\lra T_\chi\otimes\LL_i \stackrel{\gamma_i-1}{\lra} T_\chi\otimes\LL_i\lra T_\chi \lra 0$$induces an injection$$H^1_{\FF_-}(F,T_\chi\otimes\LL_i)/(\gamma_i-1) \hookrightarrow H^1_{\FF_-}(F,T_\chi).$$ We conclude using Nakayama's Lemma that $H^1_{\FF_-}(F,T_\chi\otimes\LL_i)=0$.
For each $1\leq j \leq g+1-\delta$, set $\mathcal{A}_j=(\gamma_1-1,\cdots,\gamma_j-1)$ and $\mathcal{A}_0=0$. There is an exact sequence
$$0\lra T_\chi\otimes\LL/\mathcal{A}_{j-1}\stackrel{\gamma_j-1}{\lra}  T_\chi\otimes\LL/\mathcal{A}_{j-1} \lra T_\chi\otimes\LL/\mathcal{A}_{j}\lra 0$$
that induces an injection
 $$H^1_{\FF_-}(F,T_\chi\otimes\LL/\mathcal{A}_{j-1})\Big{/} (\gamma_j-1)\hookrightarrow H^1_{\FF_-}(F,T_\chi\otimes\LL/\mathcal{A}_{j}).$$
 Noting that
 $$H^1_{\FF_-}(F,T_\chi\otimes\LL/\mathcal{A}_{g+1-\delta})=H^1_{\FF_-}(F,T_\chi)=0$$
 and using Nakayama's Lemma at each step, it follows by induction that
 $$H^1_{\FF_-}(F,T_\chi\otimes\LL/\mathcal{A}_{0})=H^1_{\FF_-}(F,T_\chi\otimes\LL)=0.$$
\end{proof}

\begin{prop}
\label{prop:4termexact}
The following sequences of $\LL$-modules are exact:
\begin{itemize}
\item[(i)]$0\lra{H^1_{\FF_{\al_\infty}}(F,\TT)}\stackrel{\textup{loc}_p^+}{\lra}\al_\infty {\lra} \left(H^1_{\FF_{-}^*}(F,\TT^*)\right)^{\vee}\lra{\left(H^1_{\FF_{\al_\infty}^*}(F,\TT^*)\right)^{\vee}}\lra 0.$
\end{itemize}
\begin{itemize}
\item[(ii)] For any class $c \in H^1_{\FF_{\al_\infty}}(F,\TT)$,
$$0\lra {\frac{H^1_{\FF_{\al_\infty}}(F,\TT)}{\LL\cdot c}}\stackrel{\textup{loc}_p^+}{\lra}\frac{\al_\infty}{\LL\cdot \textup{loc}_p^+(c)} {\lra} \left(H^1_{\FF_{-}^*}(F,\TT^*)\right)^{\vee}\lra{\left(H^1_{\FF_{\al_\infty}^*}(F,\TT^*)\right)^{\vee}}\lra 0.$$
\end{itemize}
\end{prop}

\begin{proof}
The first follows exact sequence comes from Poitou-Tate global duality, used along with Proposition~\ref{prop:minusselmervanishes}. The second is an immediate consequence of (i).
\end{proof}

\section{Rubin-Stark elements and an Euler system of rank $r$}
\label{sec:RS}

In this section, we review Rubin's \cite{ru96} integral refinement of Stark's conjectures and construct Kolyvagin systems for the modified Selmer structure $\FF_{\al_{\infty}}$ on $T_\chi\otimes\LL$, coming from the Rubin-Stark elements.

For the rest of this paper, we assume that the Rubin-Stark conjecture~\cite[Conjecture~$\textup{B}^{\prime}$]{ru96} holds.

Let $\chi,f_\chi$ and $L$ be as above, and recall the definitions of the collections of extensions $\frak{E}_0$ and $\frak{E}$ from \S\ref{subsec:notation}. Fix forever a finite set $S$ of places of $F$ that does \emph{not} contain any prime above $p$, but contains the set of infinite places $S_\infty$ and all primes $\lambda \nmid p$ at which $\chi$ is ramified. Assume that $|S| \geq r+1$. For each $\mathcal{K} \in \frak{E}$, let
$$S_{\KKK}=\{\hbox{places of } \KKK \hbox{ that lie above the places in } S\} \cup \{\hbox{places of } \KKK \hbox{ at which } \KKK/F \hbox{ is ramified}\}$$
be a set of places of $\KKK$. Let $\mathcal{O}_{\KKK,S_\KKK}^{\times}$ denote the $S_\KKK$ units of $\KKK$, and $\Delta_\KKK$ ({resp.}, $\delta_\KKK$) denote $\hbox{Gal}(\KKK/F)$ ({resp.}, $|\hbox{Gal}(\KKK/F)|$).  Rubin in~\cite[Conjecture $\textup{B}^{\prime}$]{ru96} predicts the existence of certain elements%\footnote{Note that what we call here $\tilde{\varepsilon}_{K,S_K}$ is denoted by $\varepsilon_{K,S_K}$ in~\cite{ru96,kbbstark}. For an explanation for the change of notation, see Remark \ref{rem:comparison-wild-tame-0} and Remark \ref{rem:comparison-wild-tame} below.}
$$\tilde{\varepsilon}_{\KKK,S_\KKK} \in \Lambda_{\KKK,S_\KKK} \subset \frac{1}{\delta_\KKK}{\wedge^g} \mathcal{O}_{\KKK,S_\KKK}^{\times}$$
where the module $\Lambda_{\KKK,S_\KKK}$ is defined in~\cite[\S2.1]{ru96} and has the property that for any homomorphism
$$\tilde{\psi} \in \hbox{Hom}_{\QQ_p[\Delta_\KKK]}(\wedge^g\mathcal{O}_{\KKK,S_\KKK}^{\times,\wedge} \otimes \QQ_p,\mathcal{O}_{\KKK,S_\KKK}^{\times, \wedge} \otimes \QQ_p)$$
which is induced from a homomorphism
$$\psi \in \hbox{Hom}_{\ZZ_p[\Delta_\KKK]}(\wedge^g\mathcal{O}_{\KKK,S_\KKK}^{\times,\wedge}, \mathcal{O}_{\KKK,S_\KKK}^{\times, \wedge}),$$
one has $\tilde{\psi}(\Lambda_{K,S_K}) \subset \mathcal{O}_{\KKK,S_\KKK}^{\times, \wedge}$. We remark that the $g$-th exterior power $\wedge^g\mathcal{O}_{\KKK,S_\KKK}^{\times,\wedge}$ (and all other exterior powers which appear below) is taken in the category of $\ZZ_p[\Delta_\KKK]$-modules.
\begin{rem}\label{rem:T}
Rubin's conjecture predicts that the elements $\tilde{\varepsilon}_{\KKK,S_\KKK}$ should in fact lie inside  the module $\frac{1}{\delta_\KKK}{\wedge^g} \mathcal{O}_{\KKK,S_\KKK,\mathcal{T}}^{\times}$, where $\mathcal{T}$ is a finite set of primes disjoint from $S_\KKK$, chosen in a way that the group $\mathcal{O}_{\KKK,S_\KKK,\mathcal{T}}^{\times}$ of $S_\KKK$-units which are congruent to 1 modulo all the primes in $\mathcal{T}$ is torsion-free. As explained in \cite[Remark 3.1]{kbbiwasawa}, one can safely ignore $\mathcal{T}$ as far as we are concerned in this paper.
\end{rem}
Let $F^{\textup{cyc}}$ denote the cyclotomic $\ZZ_p$-extension of $F$ and for $m\in \ZZ^+$, let  $F_{m}^{\textup{cyc}}$ be the unique subextension of $F$ of degree $p^m$.
\begin{define}
\label{def:wildStark}
For $\mathcal{K}=M\cdot L(\eta) \in \frak{E}_0$ (or $\mathcal{K}=M\cdot F(\eta) \in \frak{E}$), where $\eta \in \NN(\RR)$ and $M \subset F_\infty$ a finite extension of $F$, choose $m \in \ZZ^+$ so that $M \not\subset F_{m}^{\textup{cyc}}$ and set $M_m=M\cdot F_{m}^{\textup{cyc}}$, $\mathcal{K}_m=\KKK\cdot M_m$. Define
$$\varepsilon_{_{\KKK,S_{\KKK}}}=\mathbf{N}^r_{_{\KKK_m/\KKK}}\left(\tilde{\varepsilon}_{_{\KKK_m,S_{_{\KKK_m}}}}\right) $$
 where $\mathbf{N}^r_{_{\KKK_m/\KKK}}$ denotes the norm map induced on the $r$-th exterior power. It follows from \cite[Proposition 6.1]{ru96} that $\varepsilon_{_{\KKK,S_{\KKK}}}$ is well-defined.
%\tilde{\varepsilon}_{F(\tau),S_{F(\tau)}}$$ where $\mathbf{N}^r_{F_1(\tau)/F(\tau)}$ is the norm map induced on the $r$-th exterior power of $S_{F_1(\tau)}$-units of $F_1(\tau)$
\end{define}
As we have fixed $S$ (therefore $S_\KKK$ as well), we will often drop $S$ or $S_\KKK$ from the notation and denote $\varepsilon_{\KKK,S_\KKK}$ by $\varepsilon_{\KKK}$; or sometimes use $S$ instead of $S_\KKK$ and denote $\mathcal{O}_{\KKK,S_\KKK}$ by $\mathcal{O}_{\KKK,S}$.

For any number field $\KKK$, Kummer theory gives a canonical isomorphism
$$H^1(\KKK,\ooo(1)) \cong \KKK^{\times,\wedge} \otimes_{\ZZ_p}\ooo := \left(\varprojlim_n \KKK^{\times}/(\KKK^{\times})^{p^n})\right) \otimes_{\ZZ_p} \ooo.$$
Under this identification, we view  each $\varepsilon_{\KKK,S_\KKK}$ as an element of $\frac{1}{\delta_\KKK}\wedge^g H^1(\KKK,\ooo(1))$. The distribution relation satisfied by the Rubin-Stark elements (\cite[Proposition 6.1]{ru96}) shows that the collection $\{\varepsilon_{\KKK,S_\KKK}\}_{K\in\mathcal{K}}$ is an Euler system of rank $g$ in the sense of~\cite{pr-es}, as appropriately generalized by \cite{kbbesrankr} to allow denominators).

\subsection{Twisting by the character $\chi$}
\label{sec:twisting}
Following the formalism of~\cite[\S II.4]{r00}, we may \emph{twist} the Euler system $\{\varepsilon_{\KKK,S_\KKK}\}_{_{\KKK\in\kk}}$ of rank $g$ for the representation $\ooo(1)$, in order to obtain an Euler system  for the representation $T_\chi=\ooo(1)\otimes\chi^{-1}$.

 For a finite subextension $M$ of $F_\infty/F$ and $\eta \in \NN(\RR)$, let $\Gamma_M=\textup{Gal}(M/F)$ and define $\KKK=M\cdot F(\eta)$, $\KKK_0=M\cdot L(\eta)$. Set $G^{\eta}:=\Gal(F(\eta)/F)$,\,  $\Delta^\eta:=\Gal(L(\eta)/F)=G^\eta\times\Delta$, and finally $G^\eta_{M}:=\Gal(\KKK/F)=G^\eta\times\Gamma_M$, which is the $p$-part of $\Delta^\eta_{M}:=\hbox{Gal}(\KKK_0/F)\cong G^\eta_{M} \times \Delta=G^\eta\times\Gamma_M\times\Delta$. (These canonical factorizations of the Galois groups follow easily from the fact that  $|\Delta|$ is prime to $p$ and from ramification considerations.) The array of fields and Galois groups below summarizes this paragraph:
$$\xy\xymatrix{&M\cdot L(\eta)=\KKK_0\ar@{-}[rd]^{\Gamma_M}\ar@{-}[dl]_{\Delta}\ar@{-}[ddd]^(.33){\Delta^\eta_{M}}&\\
M\cdot F(\eta)=\KKK\ar@{-}[d]_{\Gamma_M}\ar@{-}[ddr]^(.35){G^\eta_{M}}&&L(\eta)\ar@{-}[d]^{G^\eta}\ar@{-}[ddl]_(.35){\Delta^\eta}\\
F(\eta)\ar@{-}[rd]_{G^\eta}&&L\ar@{-}[ld]^{\Delta}\\
&F&
}\endxy$$

 Let $\chi$ be as above, and let $\eee_{\chi}$ denote the idempotent $\frac{1}{|\Delta|}\sum_{\sigma \in \Delta}\chi(\sigma)\sigma^{-1}$, which regard as an element of the groups ring $\ooo[\Delta^\tau_{n}]$ via the factorization above. For simplicity, we set $\delta=\delta_{\KKK}$ (note that $\delta$ also equals $\delta_{\KKK_0}$ up to multiplication by a $p$-adic unit) allowing ourselves to be somewhat sloppy, as the denominators will not be present when the Rubin-Stark elements are utilized for our main purposes. % $\delta$ will appear as a denominator below; although it does depend on $n$ and $\tau$, we will allow ourselves to be sloppy with the notation  we use for these denominators, as they will not be present when the Stark elements are utilized for our main purposes (i.e., when they are used to construct a $\LL$-adic Kolyvagin system for the Selmer structure $\FF_{\LLL}$).

 For any integral ideal $\eta$ which is prime to $pf_{\chi}$, we define
\begin{align} \label{def:twist}
\varepsilon_{\KKK_0}^{\chi}:=\epsilon_{\chi}\varepsilon_{\KKK_0,S} &\in  \frac{1}{\delta}\epsilon_{\chi}\wedge^g H^1(\KKK_0,\ooo(1))\\ \label{eqn*}
&=\frac{1}{\delta} \wedge^g \eee_\chi H^1(\KKK_0,\ooo(1))\\
&= \frac{1}{\delta}\wedge^g H^1(\KKK_0,\ooo(1))^{\chi}.
\end{align}
%We note that the equality between the lines (\ref{def:twist}) and (\ref{eqn*}) above
%\begin{align*}
%\left(\wedge^g H^1(L_n(\tau),\ooo(1))\right)^{\chi}&=\epsilon_{\chi}\wedge^g H^1(L_n(\tau),\ooo(1))\\
%&=\wedge^g \eee_\chi H^1(L_n(\tau),\ooo(1))\\
%&=\wedge^g H^1(L_n(\tau),\ooo(1))^{\chi}
%\end{align*}
%simply holds  because $\epsilon^{r}_{\chi}=\epsilon_{\chi}$.

 Inflation-restriction yields
%\begin{equation}
%\label{eq:twist}
%H^1(\Delta,T) \lra H^1(k_n(\tau),T)\lra H^1(L_n(\tau),T)^{\Delta}\lra H^2(\Delta,T)
%\end{equation}
%where $H^1(L_n(\tau),T)^{\Delta}$ stands for the largest submodule of $H^1(L_n(\tau),T)$ on which $\Delta$ acts trivially. On the other hand, since $|\Delta|$ is prime to $p$, it follows that the very first and the very last terms in~(\ref{eq:twist}) vanish. We therefore have an isomorphism
$$H^1(\KKK,\ooo(1)\otimes\chi^{-1})\lra H^1(\KKK_0,\ooo(1)\otimes\chi^{-1})^{\Delta}.$$
On the other hand, since $G_{\KKK_0}$ is in the kernel of $\chi$,
 $$H^1(\KKK_0,\ooo\otimes\chi^{-1}) \cong H^1(\KKK_0,\ooo(1))\otimes\chi^{-1},$$ hence
$$
 %\be
 %\label{twist-isom}
H^1(\KKK,T_\chi) \stackrel{\sim}{\lra} H^1(\KKK_0,T_\chi)^{\Delta} \cong H^1(\KKK_0,\ooo(1))^{\chi}.
%\ee
$$
 This induces an isomorphism
\begin{equation}\label{eq:main twist}
\wedge^g H^1(\KKK,T) \stackrel{\sim}{\lra} \wedge^g H^1(\KKK_0,\ooo(1))^{\chi}.
\end{equation}
The inverse image of the element $\varepsilon_{\KKK_0}^{\chi}$  (which was defined in~(\ref{def:twist})) under the isomorphism induced from~(\ref{eq:main twist}) above will be denoted by $\varepsilon_{\KKK}^{\chi}$. The collection $\{\varepsilon_{\KKK}^{\chi}\}_{_{\KKK\in \frak{E}}}$ will be called the \emph{Rubin-Stark element Euler system of rank r}.

Next, we construct an \emph{Euler system of rank one} (i.e., an Euler system in the sense of~\cite{r00}) using ideas from~\cite[\S6]{ru96} and \cite[\S1.2.3]{pr-es}. The main point is that, if one applied the arguments of~\cite{ru96,pr-es} directly, all one would get (after applying Kolyvagin's descent) would be a $\LL$-adic Kolyvagin system for the coarser Selmer structure $\FF_\LL$ on $T_\chi\otimes\LL$. In Section \S\ref{subsec:ESKSmodified}, we overcome this difficulty and obtain a $\LL$-adic Kolyvagin system  for the finer Selmer structure $\FF_{\al_\infty}$ on $T_\chi\otimes\LL$.

\subsection{Choosing the homomorphisms}
\label{subsec:homs}
For any field $\KKK\in\frak{E}$, recall that $\Delta_\KKK:=\Gal(\KKK/F)$ and write $\delta=|\Delta_\KKK|$. Using the elements of
\be
\label{eqn:bighoms}
\varprojlim_{\KKK \in \frak{E}} \wedge^{r-1}\,\hbox{Hom}_{\ooo[\Delta_\KKK]}\left(H^1(\KKK,T_\chi), \ooo[\Delta_K]\right)\ee
(more precisely, using the elements those are in the image of the canonical map
\be\label{eqn:localizationhoms}{
\varprojlim_{\KKK \in \frak{E}} \wedge^{r-1}\,\hbox{Hom}_{\ooo[\Delta_\KKK]}\left(H^1_+(\KKK_p,T_\chi), \ooo[\Delta_\KKK]\right) \lra
\varprojlim_{\KKK \in \frak{E}} \wedge^{r-1}\,\hbox{Hom}_{\ooo[\Delta_\KKK]}\left(H^1(\KKK,T_\chi), \ooo[\Delta_\KKK]\right)
}
\ee
which is  induced from the localization followed by projection to $H^1_+(\KKK_p,T_\chi)$) and the Rubin-Stark elements above, we obtain an Euler system (in the sense of~\cite{r00}) for $T_\chi$, following the arguments of \cite{kbbesrankr} (which are based on Rubin's ideas \cite[\S6]{ru96}; see also~\cite[\S1.2.3]{pr-es}). We omit the details here and refer the reader to these articles. % In this section, we show how to choose the homomorphisms appearing in the inverse limit (\ref{eqn:bighoms}) carefully, so that the resulting Euler system gives rise to a Kolyvagin system for the $\al_\infty$-modified Selmer structure $\FF_{\al_{\infty}}$ on $T_\chi\otimes\LL$.
\begin{rem}
\label{rem:denomdisappear}
For $\Psi=\{\psi_{\KKK}\} \in \varprojlim_{\KKK \in \frak{E}} \wedge^{r-1}\,\hbox{Hom}_{\ooo[\Delta_\KKK]}\left(H^1(\KKK,T_\chi), \ooo[\Delta_K]\right)$ we have
$$\psi_{\KKK}(\varepsilon_{\KKK}^\chi) \in H^1(\KKK,T_\chi)$$ by the defining (integrality) property of the elements $\varepsilon_{\KKK}^\chi \in \frac{1}{\delta} \wedge^g H^1(\KKK,T_\chi),$
namely, the denominators $\delta$ will disappear once we apply the homomorphisms from~(\ref{eqn:bighoms}) on the Rubin-Stark elements.
\end{rem}

Let $\textup{ES}(T_\chi,\frak{E})=\textup{ES}(T_\chi)$ denote the collection of Euler systems for $T_\chi$ in the sense of \cite[\S2]{r00} and \cite[\S3.2]{mr02}.

\begin{define}
\label{def:alrestrictedES}
Let $\frak{L} \subset \displaystyle{\varprojlim_{M\in \frak{E}}H^1_{+}(M_{p},T_\chi)}$ be a $\LL_{\frak{K}}$-direct summand as in Definition~\ref{def:FminusonTchi}. An Euler system $\textbf{c}=\{c_\mathcal{K}\} \in \textup{ES}(T_\chi)$ is called an \emph{$\frak{L}$-restricted Euler system} if
$$\textup{loc}_p(c_\mathcal{K}) \in H^1_-(\mathcal{K}_p,T_\chi)\oplus \al_\mathcal{K}$$
for every $\mathcal{K}\in \frak{E}$. The module of $\frak{L}$-restricted Euler systems is denoted by $\textup{ES}_{\frak{L}}(T_\chi)$.
\end{define}
The following Proposition is proved following the arguments of \cite[\S 3.3-3.4]{kbbesrankr} \emph{verbatim}. The key point is to make use of Corollary~\ref{cor:localfull}.
\begin{prop} $\,$
\label{prop:ellrestrictedES}
\begin{enumerate}
\item[(i)] There exists
$\Psi=\{\psi_{\KKK}\} \in \varprojlim_{\KKK \in \frak{E}} \wedge^{r-1}\,\textup{Hom}_{\ooo[\Delta_\KKK]}\left(H^1_+(\KKK_p,T_\chi), \ooo[\Delta_\KKK]\right)$ such that $\psi_\KKK$ maps $\wedge^r H^1_+(\KKK_p,T_\chi)$ isomorphically onto $\al_\KKK$.
\item[(ii)] For $\Psi$ as above, denote its image under \textup{(}\ref{eqn:localizationhoms}\textup{)} still by $\Psi$. Then
$$\textbf{c}^{\chi}_{\Psi}:=\{\psi_{\KKK}(\varepsilon_{\KKK}^\chi)\} \in \textup{ES}_{\frak{L}}(T_\chi).$$
\end{enumerate}
\end{prop}
Let $c_{F,\Psi}^\chi:=\psi_F(\varepsilon_F^\chi)\in H^1_{\FF_{\al}}(F,T_\chi)$ be the initial term of the $\frak{L}$-restricted Euler system $\textbf{c}^{\chi}_{\Psi}$.
\subsection{From Euler systems of rank $r$ to Kolyvagin systems for modified Selmer structures}
\label{subsec:ESKSmodified}
Let $\PP_\chi$ be a fixed set of places of $F$ that does not contain the archimedean places, primes at which $T_\chi$ is ramified and primes above $p$. Recall the definition of generalized module of Kolyvagin systems $\overline{\KS}(T_\chi\otimes\LL,\FF_\LL,\PP_\chi)$ from \cite[Definition 3.1.6]{mr02}.
\begin{thm}[Mazur and Rubin]
\label{thm:ESKSmain} There is a canonical map
$$\ES(T_\chi) \lra \overline{\KS}(T_\chi\otimes\LL,\FF_\LL,\PP_\chi),$$
with the property that if $\mathbf{c}$ maps to $\pmb{\kappa} \in \overline{\KS}(\TT,\FF_{\LL},\PP)$ then
$$\kappa_1=\{c_{M}\} \in \varprojlim_{M} H^1(M,T_\chi)=H^1(F,T_\chi\otimes\LL),$$
where the inverse limit is over the finite sub-extensions $M$ of $F_{\infty}/F$.
\end{thm}
\begin{proof}
\textup Under the running hypotheses, this may be proved following the proof of Theorem 5.3.3 in \cite{mr02} line by line. %For example, for every $M\cdot F(\eta)=\KKK \in \frak{E}$, using Kolyvagin's descent one obtains a class $\kappa_{[M,\eta,I_{\eta}^\prime]} \in H^1(M,T_\chi/I_{\eta}^\prime T_\chi)$, where $I_\eta^{\prime}$ is a
\end{proof}
For $\Psi$ and $\textbf{c}^{\chi}_{\Psi}$ as in the statement of Proposition~\ref{prop:ellrestrictedES}, we let $\pmb{\kappa}^{\chi} \in \overline{\KS}(T_\chi\otimes\LL,\FF_\LL,\PP_\chi)$  be the image of $\textbf{c}^{\chi}_{\Psi}$ under the Euler systems to Kolyvagin systems map of Theorem~\ref{thm:ESKSmain}. The proof of \cite[Theorem 3.25]{kbbesrankr} shows that:

\begin{thm}
\label{thm:ESKSrestricted}
$\pmb{\kappa}^{\chi} \in \overline{\KS}(T_\chi\otimes\LL,\FF_{\al_\infty},\PP_\chi)$.
\end{thm}
\begin{rem}
\label{rem:existenceofKS}
It may be proved that the $\LL$-adic Kolyvagin system $\pmb{\kappa}^{\chi}$ which here we have constructed out of the (conjectural) Rubin-Stark elements does exist unconditionally, using the techniques of  \cite{kbb, kbbdeform}; see also~ \cite[Theorem 2.19]{kbbiwasawa}.
\end{rem}
\section{Applications to the arithmetic of CM Abelian Varieties}
\label{sec:applications}

Although our sights are set for applications of Theorem~\ref{thm:ESKSrestricted} on the study of CM abelian varieties, we first state the following two results, the latter of which may be thought of a generalization of Gras' conjecture. This result will later be used to convert all inequalities which are obtained using the Euler/Kolyvagin system machinery into equalities.

\subsection{Ideal class groups of CM fields}
For any number field $\KKK$, let $A_{\KKK}$ denote the $p$-Sylow subgroup of the ideal class group of $A_{\KKK}$.

\begin{thm}
\label{thm:gras0}
$\#H^1_{\FF_{\al}^*}(F,T_\chi^*) \leq [H^1_{\FF_{\al}}(F,T_\chi):\frak{O}\cdot c_{F,\Psi}^\chi]$.
\end{thm}

\begin{proof}
This follows from Theorem~\ref{thm:ESKSrestricted} and \cite[Theorem 5.2.14]{mr02}.
\end{proof}
The following (stronger) version of Leopoldt's conjecture was formulated by Hida and Tilouine \cite[p. 94]{ht94}:
\begin{conjecture}[$\Sigma$-Leopoldt conjecture]
\label{conj:sigmaleopoldt}
The localization map
$$\textup{loc}_p^+: \mathcal{O}_{L_\chi}^{\times,\chi} \lra\left(\prod_{\frak{P}\in\Sigma_p(L_{\chi})}\oo_{L_{\chi},\frak{P}}^{\times}\right)^{\chi}$$
is injective. Here, $\Sigma_p(L_{\chi})$ is the collection of primes of $L_\chi$ that lie above the primes in $\Sigma_p$.
\end{conjecture}
\begin{rem}
\label{rem:motivatesigmaleo}
As explained in 1.2.1 of \cite{ht94}, both the $\Sigma$-Leopoldt conjecture and the Leopoldt conjecture itself follow from the $p$-adic Schanuel's conjecture.
\end{rem}
We assume until the end the following is true (as well as the truth of Rubin-Stark conjectures):
\begin{itemize}
\item[$\mathbf{H.\Sigma L}$] $\Sigma$-Leopoldt conjecture holds true for the field $L_\chi$.
\item[$\textbf{H.S}$] The set $S$ that appears in the definition of Rubin-Stark elements (see the start of \S\ref{sec:RS}) contains no archimedean places of $F$ that splits in $L_\chi/F$.
\end{itemize}

\begin{thm}
\label{thm:gras1}
\begin{itemize}
\item[(i)] The $\frak{O}$-module $H^1_{\FF_{\al}}(F,T_\chi)$ is free of rank one and $H^1_{\FF_{\al}^*}(F,T_\chi^*)$ is finite.
\item[(ii)] $\# A_{L_\chi}^\chi \leq [\wedge^g\, \mathcal{O}_{L_\chi}^{\times,\chi}\,:\,\frak{O}\cdot \varepsilon_{F}^\chi]$\,.
\item[(iii)] If the inequality in \textup{(ii)} is sharp then so is the inequality in  Theorem~\ref{thm:gras0}.
\end{itemize}
\end{thm}
\begin{proof}
(i) is proved using Theorem~\ref{thm:gras0} and following the proof of \cite[Corollary 3.6]{kbbstark}, and (ii) following the proof of \cite[Theorem 3.10]{kbbstark}. The key point in the proofs we refer to in \cite{kbbstark} is that, the map $\iota$ in loc.cit. from the group of units to semi-local units is injective as Leopoldt's conjecture was assumed to hold. To apply the arguments in loc.cit., one needs to replace the map $\iota$ by  the map
$$\textup{loc}_p^+: \mathcal{O}_{L_\chi}^{\times,\chi} \lra\left(\prod_{\frak{P}\in\Sigma_p(L_{\chi})}\oo_{L_{\chi},\frak{P}}^{\times}\right)^{\chi}=H^1_{+}(F_p,T_\chi),$$
where the last equality follows from the proof of \cite[Proposition 2.6]{r00} as we have assumed (\ref{eqn:assnotrivxhi}). The injectivity of $\textup{loc}_p^+$ is the statement of our hypothesis $\mathbf{H.\Sigma L}$.%This is what we verify now. Let $\Sigma_{L_\chi}$ denote the collection of infinite places of $L_\chi$ that lie above $\Sigma$ and let $\Sigma_{L_\chi}^\prime\subset \Sigma_{L_\chi}$ be any subset missing exactly one element $\sigma_0 \in\Sigma_{L_\chi}$. %Let $\frak{P}_0$ be the prime of $L_\chi$ corresponding to the embedding $\iota_p\circ \sigma_0$ and let $\Sigma_p(L_\chi)^\prime= \Sigma_p(L_\chi)-\{\frak{P}_0\}$.
%Set $\frak{r}=\#\Sigma_{L_\chi}^\prime$ and let $\frak{U}=\{\frak{u}_1,\cdots,\frak{u}_\frak{r}\}$ be a set of fundamental system of units. The assumed hypothesis \textbf{H.LC} implies that  the $p$-adic regulator
%$$R_p(L_\chi)=\det\left(2\log_p(\iota_p\circ \sigma(\frak{u}))\right)_{\frak{u}\in \frak{U};\, \sigma \in \Sigma_{L_\chi}^\prime}$$
%is non-zero. This in turn shows that the image of the map
%$$\oo_{L_\chi}^{\times,\wedge}\lra \prod_{\frak{P}\in \Sigma_p({L_{\chi}})}\oo_{L_{\chi},\frak{P}}^{\times,\wedge}$$
%has rank $\frak{r}$. Using the fact that the ring $\ooo[\Delta]$ is semi-simple, it follows that %char by char, the upper bound is \frak{r}_\chi=rank of O_{L_\chi}^{\times,\chi}
%the image of the map $\textup{loc}_p^+$ is full rank. Since we assumed (\ref{eqn:chiisnotteich1}), we have $\chi\neq \omega$ and thus $\mathcal{O}_{L_\chi}^{\times,\chi}$ is torsion-free. It therefore follows that the map $\textup{loc}_p^+$ is injective, as desired.

(iii) also follows from the proofs of Corollary 3.9 and Theorem 3.10 of loc.cit.
\end{proof}
Choosing the auxiliary set of primes $\mathcal{T}$ that appears in the definition of Rubin-Stark elements carefully (as in \cite[\S 2.1]{kbbstark}, see also the discussion preceding Theorem 3.11 in loc.cit.), one may use the analytic class number formula for all the fields between $L_\chi$ and $F$ to convert the inequality of Theorem~\ref{thm:gras1}(ii) (and therefore  also in Theorem~\ref{thm:gras0}) into an equality. See \cite[\S5]{ru92} and \cite[\S4.2]{popescu} for further details.
\begin{cor}
\label{cor:gras}
\begin{itemize}
\item[(i)] $\# A_{L_\chi}^\chi = [\wedge^g\, \mathcal{O}_{L_\chi}^{\times,\chi}\,:\,\frak{O}\cdot \varepsilon_{F}^\chi]$\,.
\item[(ii)] $\#H^1_{\FF_{\al}^*}(F,T_\chi^*) = [H^1_{\FF_{\al}}(F,T_\chi):\frak{O}\cdot c_{F,\Psi}^\chi]$
\end{itemize}
\end{cor}
\subsection{Iwasawa theory}
\label{subsec:IwTheory}
Let $c^\chi_{F_\infty,\Psi}:=\{c_{M,\Psi}^\chi\} \in \varprojlim_M H^1(M,T_\chi)$, %=H^1(F,T_\chi\otimes\LL)$
where the inverse limit is over finite subextensions $M$ of $F_\infty/F$. Recall the Selmer structure $\FF_-$ on $T_\chi\otimes\LL$, defined as in \S\ref{subsec:compareselmer}.

Let $$ \textup{loc}_p^{+,\,\otimes r}: \wedge^r H^1(F,X) \lra \wedge^r H^1_{+}(F_p,X)$$
(where $X=T_\chi$ or $T_\chi\otimes\LL$) be the map induced  from  $\textup{loc}_p^+$. Define
$$\textup{loc}_p^{+,\,\otimes r}\left(\varepsilon_{F_\infty}^\chi \right):=\{\textup{loc}_p^{+,\,\otimes r}\left(\varepsilon_{M}^\chi \right)\}_{M} \in \varprojlim_M \wedge^r H^1_+(M,T_\chi)=\wedge^r H^1(F,T_\chi\otimes\LL),$$
where the inverse limit is with respect to finite subextensions $M$ of $F_\infty/F$, and the last equality holds true since  the $\ooo[\Gamma_M]$-module $H^1_+(M,T_\chi)$ is free of rank $r$, by Proposition~\ref{prop:localrank}.
\begin{thm}
\label{thm:mainconjforTchi} Under the running assumptions,
\begin{itemize}
\item[(i)] We have
$$\#H^1_{\FF_-^*}(F,T_\chi)=\#\al/\ooo\cdot \textup{loc}_p^+\left({c}_{F,\Psi}^\chi \right)= \#\wedge^rH^1_{+}(F_p,T_\chi)/\ooo\cdot \textup{loc}_p^{+,\,\otimes r}\left(\varepsilon_{F}^\chi \right),$$
and all these quantities are finite.%The module $H^1_{\FF^*_{\al_\infty}}(F,(T_\chi\otimes\LL)^*)$ is $\LL$-cotorsion.
\item[(ii)] $\textup{char}\left(H^1_{\FF^*_{\al_\infty}}(F,(T_\chi\otimes\LL)^*)^\vee\right) \mid \textup{char}\left(H^1_{\FF_{\al_\infty}}(F,T_\chi\otimes\LL)/\LL\cdot c_{F_\infty,\Psi}^\chi\right)$.
\item[(iii)] The module $H^1_{\FF^*_{-}	}(F,(T_\chi\otimes\LL)^*)$ is $\LL$-cotorsion and
$$\textup{char}\left(H^1_{\FF^*_{-}	}(F,(T_\chi\otimes\LL)^*)^\vee\right)=
\textup{char} \left(\wedge^r H^1_{+}(F_p,T_\chi\otimes\LL)/\LL\cdot  \textup{loc}_p^{+,\,\otimes r}\left(\varepsilon_{F_\infty}^\chi \right)\right)$$
\end{itemize}
\end{thm}
\begin{proof}
%Let $\mathcal{A}=\ker\{\LL \ra \ooo\}$ be the augmentation ideal. By \cite[Lemma 3.5.3]{mr02}, we may identify $H^1_{\FF^*_{\al_\infty}}(F,(T_\chi\otimes\LL)^*)[\mathcal{A}]$ with $H^1_{\FF^*_{\al_\infty}}(F,(T_\chi\otimes\LL)^*[\mathcal{A}])$, and
%\begin{align*}H^1_{\FF^*_{\al_\infty}}(F,(T_\chi\otimes\LL)^*[\mathcal{A}])&= H^1_{\FF^*_{\al_\infty}}(F,(T_\chi\otimes\LL/\mathcal{A})^*)\\
%&=H^1_{\FF^*_{\al_\infty}}(F,T_\chi^*).
%%end{align*}
%Since $H^1_{\FF^*_{\al_\infty}}(F,T_\chi^*)$ is finite, (i) follows. When the Iwasawa algebra $\LL$ has Krull dimension 2, (ii) follows from \cite[Theorem 5.3.1]{mr02}. %This theorem gives only a divisibility of RHS by the LHS. To get to an equality, proceed as in the "NOTES", do not yet type it.
%The general case reduces to the case of dimension 2 using \cite[\S3]{ochiaideform}. The first equality in (iii) follows from (ii) and Proposition~\ref{prop:4termexact}(ii); and the second equality from the choice of $\Psi$ as in Proposition~\ref{prop:ellrestrictedES}, which induces an isomorphism
%$$\wedge^r H^1_+(M,T_\chi) \lra \al_M$$
%for every finite subextension $M$ of $F_\infty/F$.
The second equality in (i) is deduced using the defining property of $\Psi$, see Proposition~\ref{prop:ellrestrictedES}(i).

As in Proposition~\ref{prop:4termexact}\footnote{See also [Rub00] Theorem I.7.3,
proof of Theorem III.2.10 and \cite[\S III.1.7]{deshalit87}.}, the Poitou-Tate global duality and Lemma~\ref{lem:minusvanishingforunits} yields an exact sequence
$$0\ra H^1_{\FF_{\al}}(F,T_\chi)/\ooo\cdot c_{F,\Psi}^\chi \lra \al/\ooo\cdot \textup{loc}_p^+\left(c_{F,\Psi}^\chi\right)  \lra H^1_{\FF_{-}^*}(F,T_\chi^*) \lra H^1_{\FF_{\al}^*}(F,T_\chi^*)\ra 0$$
The first equality in (i) now follows from Corollary~\ref{cor:gras}(ii).%and the finiteness from the fact that $\al$ is free of rank one, loc_p^+ is injective and c_{F,\Psi}^\chi is non-zero. (by LC).

When the Iwasawa algebra $\LL$ has Krull dimension 2, (ii) follows from \cite[Theorem 5.3.1]{mr02}. %This theorem gives only a divisibility of RHS by the LHS. To get to an equality, proceed as in the "NOTES", do not yet type it.
The general case is reduced to the case of dimension 2 applying the results of Ochiai in \cite[\S3]{ochiaideform}.

Let $\mathcal{A}=\ker\{\LL \ra \ooo\}$ be the augmentation ideal. By \cite[Lemma 3.5.3]{mr02}, we may identify $H^1_{\FF^*_{-}}(F,(T_\chi\otimes\LL)^*)[\mathcal{A}]$ with $H^1_{\FF^*_{-}}(F,(T_\chi\otimes\LL)^*[\mathcal{A}])$, and
\begin{align}
\label{eqn:modaugdescent} H^1_{\FF^*_{-}}(F,(T_\chi\otimes\LL)^*[\mathcal{A}])&= H^1_{\FF^*_{-}}(F,(T_\chi\otimes\LL/\mathcal{A})^*)\\
\notag&=H^1_{\FF^*_{-}}(F,T_\chi^*).
\end{align}
Since $H^1_{\FF^*_{-}}(F,T_\chi^*)$ is finite by (i), the first statement of (iii) follows. It also follows from (ii) and Proposition~\ref{prop:4termexact}(ii) (along with the choice of $\Psi$ as in Proposition~\ref{prop:ellrestrictedES}%, which induces an isomorphism $$\wedge^r H^1_+(M,T_\chi) \lra \al_M$$ for every finite subextension $M$ of $F_\infty/F$.
) that
$$\textup{char}\left(H^1_{\FF^*_{-}	}(F,(T_\chi\otimes\LL)^*)^\vee\right) \mid \textup{char} \left(\wedge^r H^1_{+}(F_p,T_\chi\otimes\LL)/\LL\cdot  \textup{loc}_p^{+,\,\otimes r}\left(\varepsilon_{F_\infty}^\chi \right)\right).$$
It remains to prove that this divisibility may in fact be turned into an equality and this is what we carry out in what follows.

We will first check the equality modulo the augmentation ideal $\mathcal{A}$ (namely, the statements (\ref{eqn:lhsmodA}) and (\ref{eqn:rhsmodA})) , which using Lemma~\ref{lem:charidealprimetoA}, Lemma~\ref{lem:equalitymodAimpliesequality} below and the divisibility obtained above will conclude the proof.

Observe that the $\LL$-module $\al_\infty/\LL\cdot \textup{loc}_p^{+}(c_{F_\infty,\Psi}^\chi)$ is cyclic. We see therefore that
\begin{align*}
\textup{char} \left(\wedge^r H^1_{+}(F_p,T_\chi\otimes\LL)/\LL\cdot  \textup{loc}_p^{+,\,\otimes r}\left(\varepsilon_{F_\infty}^\chi \right)\right)&=\textup{char} \left(\al_\infty/\LL\cdot \textup{loc}_p^{+}(c_{F_\infty,\Psi}^\chi)\right)\\
&=\textup{Fitt}_{\LL}\left( \al_\infty/\LL\cdot \textup{loc}_p^{+}(c_{F_\infty,\Psi}^\chi)\right), \end{align*}
where the first equality is obtained thanks to the choice of $\Psi$ and $\textup{Fitt}_{\LL}\left(M\right)$ denotes the initial Fitting ideal of a $\LL$-module $M$. Thus,
\begin{align}\label{eqn:lhsmodA}[\ooo:\textup{char}(\al_\infty/\LL\cdot c_{F_\infty,\Psi}^\chi)\otimes_{\LL}\LL/\mathcal{A}]&=[\ooo:\textup{Fitt}_{\ooo}\left((\al_\infty/\LL\cdot c_{F_\infty,\Psi}^\chi)\otimes_{\LL}\LL/\mathcal{A}\right)]\\
\notag&=[\ooo:\textup{Fitt}_{\ooo}(\al/\ooo\cdot c_{F,\Psi}^\chi)]\\
\notag&=\#H^1_{\FF_-^*}(F,T_\chi^*).\end{align}
We next check that
\begin{align}\label{eqn:rhsmodA}
[\ooo:\textup{char}\left(H^1_{\FF^*_{-}}(F,(T_\chi\otimes\LL)^*)^\vee\right)\otimes_{\LL}\LL/\mathcal{A}]&=\#H^1_{\FF^*_{-}}(F,(T_\chi\otimes\LL)^*)^\vee\otimes\LL/\mathcal{A}\\
\notag &=\#H^1_{\FF_-^*}(F,T_\chi^*).
\end{align}
where the second equality holds thanks to (\ref{eqn:modaugdescent}). In order to achieve this, we appeal to Nekov\'a\v{r}'s theory of Selmer complexes and the descent formalism built in his theory.

The Selmer complex $\widetilde{R\Gamma}_{f,\textup{Iw}}(F_\infty/F,T_\chi)$ (resp., the dual complex $\widetilde{R\Gamma}_{f}(F_{\Sigma(\FFc)}/F_\infty,T_\chi^*)$, in the sense of \cite[Proposition 9.7.2]{nekovar06}) related to the Selmer group $H^1_{\FF_-^*}(F,(T_\chi\otimes\LL)^*)$ is given by the Greenberg local conditions determined by
$$U_v^+(T_\chi)= \begin{cases}
0, &  \hbox{ if } v \in \Sigma_p,\\
T_\chi,            &               \hbox{ if } v \in \Sigma_p^c. \end{cases} $$
(resp., $$U_v^+(T_\chi^*)= \begin{cases}
T_\chi^*, &  \hbox{ if } v \in \Sigma_p,\\
0,            &               \hbox{ if } v \in \Sigma_p^c. \end{cases}$$
for the dual complex).
Since we assume (\ref{eqn:assnotrivxhi}), \cite[Lemma 9.6.3]{nekovar06} (and \cite[Proposition 8.8.6]{nekovar06} to pass to limit) shows that
$$\tilde{H}^1_{f}(F_{\Sigma(\FFc)}/F_{\infty},T_\chi^*) \stackrel{\sim}{\lra} H^1_{\FF_-^*}(F,(T_\chi\otimes\LL)^*),$$
and \cite[8.9.6.2]{nekovar06} that
$$\tilde{H}^2_{f,\textup{Iw}}(F_\infty/F, T_\chi)\cong \tilde{H}^1_{f}(F_{\Sigma(\FFc)}/F_{\infty},T_\chi^*)^\vee.$$
Furthermore, we proved above that the $\LL$-module
$$\tilde{H}^1_{f}(F_{\Sigma(\FFc)}/F_{\infty},T_\chi^*)\cong  H^1_{\FF_-^*}(F,(T_\chi\otimes\LL)^*)$$
is cotorsion. Set $A_\chi=T_\chi\otimes\frak{F}/\ooo$ and fix a prime $\wp$ of $F$ and a prime $\tilde{\wp}$ of $L_\chi$ above $\wp$. Let $\Delta_\wp$ be the decomposition group of $\tilde{\wp}$ in $\Delta$. Since we assumed (\ref{eqn:assnotrivxhi}) and (\ref{eqn:chiisnotteich1}), we observe that (see \cite[\S9.5]{nekovar06}) %(\ref{eqn:chiisnotteich}) shows that $\chi$ is not the Teichmuller character, which proves the first displayed equality below
\be\tilde{H}^0_f(F,A_\chi)=\pmb{\mu}_{p^\infty}(L_\chi)^\chi=0\ee
\be\label{eqn:h3tildevanishes}\tilde{H}^3_f(F,T_\chi)=(\ooo[\Delta/\Delta_\wp])^\chi=0\ee
hence by duality that
$$\tilde{H}^0_f(F,X)=0$$
for $X=A_\chi, T_\chi^*$. The proof of \cite[Proposition 9.7.7]{nekovar06} thus shows that the complex $\widetilde{R\Gamma}_{f,\textup{Iw}}(F_\infty/F,T_\chi)$ may be represented by a complex
$$\textup{Cone}\left(M \stackrel{u}{\lra} M\right)[-2]$$
where $M$ is a free $\LL$-module of finite type and $u$ is injective. This, together with Nekov\'a\v{r}'s control theorem \cite[Proposition 8.10.1]{nekovar06}
$$\widetilde{R\Gamma}_{f,\textup{Iw}}(F_\infty/F,T_\chi)\stackrel{\mathbf{L}}{\otimes}_{\LL} \ooo\stackrel{\sim}{\lra} \widetilde{R\Gamma}_{f}(F,T_\chi)$$
concludes the proof of (\ref{eqn:rhsmodA}), hence also the proof of the theorem.%char ideal of \tilde{H}^2_{f,\textup{Iw}} is the det_{\LLL}(u), therefore char ideal mod augmentation ideal is det_{\LL/\mathcal{A}}(\tilde{H}^2_f).
\end{proof}
\begin{rem}
\label{rem:selmercomplexconcentratedinonedegree}
One may also observe directly by Proposition \ref{prop:minusselmervanishes} that
$$\tilde{H}^1_{f,\textup{Iw}}(F_\infty/F,T_\chi)\cong H^1_{\FF_-}(F,T_\chi\otimes\LL)=0$$
and thus the Selmer complex $\widetilde{R\Gamma}_{f,\textup{Iw}}(F_\infty/F,T_\chi)$ has no cohomology in degree 1. Also, the vanishing $(\ref{eqn:h3tildevanishes})$ implies by Nakayama's Lemma (using \cite[8.10.3.3]{nekovar06}) that the Selmer complex has no cohomology in degree 3 either.
\end{rem}
\begin{lemma}
\label{lem:charidealprimetoA}
$\textup{char}\left(H^1_{\FF^*_{-}}(F,(T_\chi\otimes\LL)^*)^\vee\right) \not\subset \mathcal{A}$.
\end{lemma}
\begin{proof}
This follows from (\ref{eqn:rhsmodA}) and the finiteness of $H^1_{\FF^*_{-}}(F,T_\chi^*)$.
%By \cite[Chapter 7, \S4.6]{bourbakicommalg}, it follows that
\end{proof}
\begin{lemma}
\label{lem:equalitymodAimpliesequality}
Suppose $f,g \in \LL$ are such that $f \mid g$, $f - g \in \mathcal{A}$ and $f\notin\mathcal{A}$. Then $f/g\in \LL^\times$.
\end{lemma}
\begin{proof}
Write $g=f\cdot h$ with $h\in \LL$, so that $f-g=f(1-h) \in \mathcal{A}$. Since $f\notin \mathcal{A}$, it follows that $1-h \in \mathcal{A} \subset \frak{m}_{\LL}$, where $\frak{m}_{\LL}$ is the maximal ideal. Hence $h$ is indeed a unit.
\end{proof}
\begin{rem}
\label{rem:bigselmerexplicit}
Mimicking the proof of \cite[Proposition 3.2.6]{r00} (with the aid of the assumption (\ref{eqn:assnotrivxhi})), we see that
$$H^1_{+}(F_p,T_\chi\otimes\LL)=\varprojlim_{M} \bigoplus_{\wp \in \Sigma_p}\oo_{ML_\chi,\wp}^{\times,\chi}$$
where $\oo_{ML_\chi,\wp}=\oo_{ML_\chi}\otimes \oo_{F_{\wp}}$ and the inverse limit is over finite subextensions of $F_\infty/F$. Similarly (using again the arguments of \cite[\S1.6]{r00}), the $\LL$-module $H^1_{\FF_-^*}(F,(T_\chi\otimes\LL)^*)$ may be identified by $\textup{Gal}(M_\infty/L_\infty)^{\chi}$, where $L_\infty=L_\chi F_\infty$ and $M_\infty$ is the maximal abelian extension of $L_\infty$ unramified outside $\Sigma_p$. Theorem~\ref{thm:mainconjforTchi}(iii) may therefore be regarded as a natural generalization of \cite[Theorem 4.1]{rubinmainconj}.
\end{rem}
\subsection{Katz's $p$-adic $L$-function}
\label{subsec:conntoKatz}
Attached to the $p$-ordinary CM-type $\Sigma$ and the character $\chi$, Katz \cite{katz78} and Hida-Tilouine~\cite[Theorem II]{ht93} has constructed a $p$-adic $L$-function $\mathcal{L}_{\chi}^{\Sigma} \in \LL_{\mathcal{W}}:=\mathcal{W}_\ooo[[\Gamma]]$, where $\mathcal{W}_\ooo$ is the composite of $\ooo$ and $\mathcal{W}$ (and $\mathcal{W}$ is the $p$-adic completion of the ring of integers of the maximal unramified extension of $\QQ_p$), that $p$-adically interpolates the algebraic parts (in the sense of \cite{shimura75}) of the critical Hecke $L$-values for $\chi$ twisted by the characters of $\Gamma$. In the spirit of the main result of \cite{yager}, we propose the following conjecture under our running assumptions:
\begin{conj}
\label{conj:yager}
$ \textup{char} \left(\wedge^g H^1_{+}(F_p,T_\chi\otimes\LL_{\mathcal{W}})/\LL_{\mathcal{W}}\cdot  \textup{loc}_p^{+,\,\otimes g}\left(\varepsilon_{F_\infty}^\chi \right)\right)=\left(\mathcal{L}_{\chi}^{\Sigma}\right)$.
\end{conj}
In view of Theorem~\ref{thm:mainconjforTchi}(iii), this conjecture is equivalent to the following statement:
\begin{thm}
\label{thm:mainconj}
Conjecture 1 holds true if and only if the Katz $p$-adic $L$-function $\mathcal{L}_{\chi}^{\Sigma}$ generates $\textup{char}\left(H^1_{\FF^*_{-}}(F,(T_\chi\otimes\LL_{\mathcal{W}})^*)^\vee\right).$
\end{thm}
The latter statement is known as the $(g+1)$-variable\footnote{As we assumed \textbf{(H.LC)}, the Iwasawa algebra $\LL=\ooo[[\Gamma]]$ is isomorphic to a power series ring in $g+1$ variables.} main conjecture, see \cite[Page 90]{ht94}. Building on the works of Hida-Tilouine~\cite{ht94}, Hida~\cite{hida06, hidaquadratic}, Mainardi~\cite{mainardi}, Hsieh has recently proved the CM main conjecture in \cite{hsiehCMmainconj} under the following hypotheses:

(\textbf{H.1}) $p>5$ is prime to the minus part of the class number of $F$, to the order of $\chi$ and is unramified in $K/\QQ$.

(\textbf{H.2}) $\chi$ is unramified in $\Sigma_p^c$ and $\chi\omega^{-a}$ is unramified at $\Sigma_p$ for some integer $a\not\equiv 2 \mod p-1$.

(\textbf{H.3}) $\chi$ is \emph{anticyclotomic} in the sense that $\chi(c\delta c^{-1})=\chi(\delta)^{-1}$ for $\delta\in \Delta$ and $c \in G_K$ that induces the generator of $\textup{Gal}(F/K)$.%G_F is normal in G_K so c\deltac^{-1} still lies in G_F.

(\textbf{H.4}) $\chi(\wp)\neq1$ for any $\wp\in \Sigma_p$. (Compare to (\ref{eqn:assnotrivxhi}))

(\textbf{H.5}) The restriction of $\chi$ to $G_{F(\sqrt{p^*})}$ (where $p^*=(-1)^{\frac{p-1}{2}}p$) is non-trivial.

The $g$-variable \emph{anticyclotomic} main conjecture (and therefore the $g$-variable version of Conjecture 1 above) may be verified under much less restrictive hypothesis, namely assuming only \textbf{H.3-5}. See \cite[Corollary 2]{hidaquadratic}.

%\subsection{Hecke characters}
%For a continuous character $\rho$ of $\Gamma$, let  $\textup{Tw}_{\rho}: \LL \ra \LL$ be the $\ooo$-linear homomorphism induced by $\gamma \mapsto \rho(\gamma)\gamma$ for $\gamma \in \Gamma$. Let $\ooo_\rho$ denote the free $\ooo$-module rank one on which $G_F$ acts via $G_F\twoheadrightarrow\Gamma\stackrel{\rho}{\ra}\ooo^\times$. Fix any generator $\xi_\rho$ of $\ooo_{\rho}$. Suppose $\chi=\omega_\psi$ and let  $\epsilon_\psi \in \wedge^r H^1_+(F_p,\TT)$ be the image of $\textup{loc}_p^+(\varepsilon_{F_\infty}^\chi) \otimes \xi_{\langle \psi \rangle^{-1}}$ under the isomorphism
%$$\wedge^r H^1_+(F_p,T_\chi)\otimes \ooo_{\langle \psi \rangle^{-1}}\stackrel{\sim}{\lra} \wedge^r H^1_+(F_p,\TT)$$
%induced from (\ref{eqn:localtwisted}). As an immediate consequence of Theorem~\ref{thm:mainconjforTchi} and Rubin's formal twisting argument \cite[\S VI]{r00}, we conclude the following:
%\begin{thm}
%\label{thm:mainconjforheckechars}
%The module $H^1_{\FF^*_{-}}(F,\TT^*)$ is $\LL$-cotorsion and
%$$\textup{char}\left(H^1_{\FF^*_{-}	}(F,\TT^*)^\vee\right)=
%\textup{char} \left(\wedge^r H^1_{+}(F_p,\TT)/\LL\cdot \epsilon_\psi\right)$$
%Furthermore, if Conjecture 1 holds true, then
%$$\textup{Tw}_{\langle \psi \rangle^{-1}}\left(\textup{char}\left(H^1_{\FF^*_{-}	 }(F,\TT^*)^\vee\right)%\right)=(\al_{\chi}^{\Sigma}).$$
%\end{thm}

\subsection{Hecke characters attached to CM abelian varieties}
\label{subsec:HeckecharCMab}
We start this subsection with an overview of well-known facts about CM abelian varieties that we need below, which are originally due to Serre-Tate and Shimura. Let $A_{/K}$ be an abelian variety which has CM by $F$. We assume that $\textup{End}_F(A)=\oo_F$; however, the arguments in this section will carry out to the more general case when the index of the order $\textup{End}_F(A)$ inside the maximal order $\oo_F$ is assumed to be prime to $p$. Assume also that the field $F$ contains no nontrivial $p$-th root of unity.

Let $T_p(A)=\varprojlim A[p^n]$ be the $p$-adic Tate-module of $A$. It is a free $\ZZ_p$-module of rank $2g$ on which $G_F$ acts continuously. As explained in the Remark on page 502 of \cite{serretate}, $T_p(A)$ is free of rank one over $\oo_F\otimes\ZZ_p=\prod_{\frak{p}} \oo_\frak{p}$, where the product is over the primes of $F$ that lie above $p$. This yields a decomposition $T_p(A)=\bigoplus_{\frak{p}}T_\frak{p}(A)$, where each $T_\frak{p}(A)=\varprojlim A[\frak{p}^n]$ is a free $\oo_\frak{p}$-module of rank one. The $G_F$-action on $T_p(A)$ gives rise to a character
$$\psi_\frak{p}: G_F \lra \oo_\frak{p}^\times.$$
By \cite[\S 2]{ribetcompositio}, $\psi_\frak{p}$ is surjective for $p$ large enough; we fix until the end a prime $p$ satisfying this condition. We thence obtain a decomposition
$$T_p(A)\otimes_{\ZZ_p}\overline{\QQ}_p=\bigoplus_{\frak{p}\mid p}\bigoplus_{\sigma: F_\frak{p} \hookrightarrow \overline{\QQ}_p} V_{\frak{p}}^\sigma$$
where $V_{\frak{p}}^\sigma$ is the one-dimensional $\overline{\QQ}_p$-vector space on which $G_F$ acts via the character
$\psi_\frak{p}^\sigma$, which is the compositum
$$G_F\stackrel{\psi_\frak{p}}{\lra} \oo_\frak{p}^\times \stackrel{\sigma}{\lra} \overline{\QQ}_p$$
Fix an embedding $j_\infty:\overline{\QQ}\hookrightarrow \mathbb{C}$ and $j_p:\overline{\QQ}\hookrightarrow \mathbb{C}_p$ extending $\iota_p$. Let $\frak{J}=\Sigma\cup\Sigma^c$ be the set of all embeddings of $F$ into $\overline{\QQ}$.
Attached to $A$, there is a character
$$\pmb{\psi}: \mathbb{A}_F/F^\times\lra F^{\times},$$
which induces the Gr\"ossencharacters
$$\psi_\tau=j_\infty\circ\tau\circ \pmb{\psi}:  \mathbb{A}_F/F^\times \lra \mathbb{C}^\times$$
and its $p$-adic avatars
$$\psi_\tau^{(p)}=j_p\circ\tau\circ\pmb{\psi}:\mathbb{A}_F/F^\times \lra \mathbb{C}_p^\times.$$
Theory of complex multiplication identifies the two sets $\{\textup{rec}\circ \psi_\tau^{(p)}\}_{\tau\in \frak{J}}$ and $\{\psi_\frak{p}^\sigma\}_{\frak{p},\sigma}$ of $p$-adic Hecke characters, where $\textup{rec}:\mathbb{A}_F/F^\times\ra G_F$ is the reciprocity map.
The Hasse-Weil $L$-function $L(A/F,s)$ of $A$ then factors into a product of Hecke $L$-series
$$L(A/F,s)=\prod_{\tau\in\frak{J}} L(\psi_\tau,s).$$
We assume henceforth that $A$ is principally polarized. Fix $\varepsilon \in \Sigma$ and identify $F$ by $F^\varepsilon$. This choice in turn fixes a prime $\wp \in \Sigma_p$ and $\sigma: F_{\wp} \hookrightarrow \overline{\QQ}_p$ in a way that $\textup{rec}\circ\psi_\varepsilon^{(p)}=\psi_\wp^\sigma$. Set $\ooo:=\sigma(\oo_{F_{\wp}})$ and let $\frak{F}:=\textup{Frac}(\ooo)$ denote the fraction field of $\ooo$. Define
$$\psi:=\psi_{\wp}^\sigma: G_F \twoheadrightarrow \ooo^\times;$$
this is the Hecke character for which we apply the results from \S\ref{subsec:conntoKatz}. Note in particular that we have $T^*\cong A[\varpi^\infty]$.
For $\wp$ as above, we assume the following non-anomaly condition on $A$:
\be\label{eqn:hna}   A(F_v)[\varpi]=0 \hbox{  for every prime } v  \hbox{ of } F \hbox{ above } p. \ee
For the character $\psi$, the condition (\ref{eqn:chiisnotteich}) holds trivially true. %; and also condition (\ref{eqn:assumepsiisnotcyclo}) by weight considerations.
The hypothesis (\ref{eqn:assnotrivxhi}) holds true for $\chi=\omega_\psi$ since we assumed (\ref{eqn:hna}). As we have also assumed that the CM field $F$ contains no $p$-th roots of unity, it follows that the Teichm\"uller character $\omega$ is totally ramified at all primes above $p$; and as $\omega_\psi$ is ramified at only $\wp$, the condition (\ref{eqn:chiisnotteich1}) for $\chi=\omega_\psi$ is also satisfied.

%Let $\psi^*=\psi^{-1}\otimes \chi_{\textup{cyc}}$ be the Cartier dual of $\psi$. %Given by, in truth, $\psi^*(x)=\psi^{-1}(x^c)\otimes \chi_{\textup{cyc}}(x)$
%Then $\psi^*$ corresponds to the embedding $\varepsilon^c \in \Sigma^c$ (and the prime $\wp^c=\wp \in \Sigma_p^c$) and local Tate-duality shows that it gives the action on the $\wp^c$-adic Tate-module $T_{\wp^c}(A)$. As for the Gr\"ossencharacters, we have $\psi_{\epsilon^c}(x)=\psi_\varepsilon(x^c)^{-1}||x||,$
%where $||\cdot||$ stands for the adelic norm.
\begin{define}
\label{def:selmergroupsforE}
For $X=F$ or $F_\infty$, let $\textup{Sel}_\varpi(A/X)$ denote the classical $\varpi$-adic Selmer group attached to $A$. Define the \emph{$\varpi$-relaxed Selmer group} by setting %$\textup{Sel}_\varpi^{\,\prime}(A/X)$ by setting
$$\textup{Sel}_\varpi^{\,\prime}(A/X):=\ker\left(H^1(X,A[\varpi^\infty])\lra\prod_{v\nmid \varpi} \frac{H^1(X_v,A[\varpi^\infty])}{\textup{im}\left(A(X_v)\otimes \frak{F}/\ooo\stackrel{\kappa_v}{\hookrightarrow} H^1(X_v,A[\varpi^\infty]) \right)} \right),$$
where $\kappa_v$ is the Kummer map.
\end{define}
Obviously, $\textup{Sel}_\varpi(A/X) \subset \textup{Sel}^{\,\prime}_\varpi(A/X)$.  %Let $\FF_{\textup{BK}}$ (resp., $\FF_{\textup{BK}}^*$) denote the Bloch-Kato Selmer structure on $T$ (resp., on $T^*$), defined as in \cite{bk}. We note that $H^1_{\FF_{\textup{BK}}^*}(F,T^*)=\textup{Sel}(A/F)$, c.f., \cite[????]{bk}.
Recall the $\LL$-module $\frak{X}_\infty=\textup{Gal}(M_\infty/L_\infty)^{\chi}$ (with $\chi=\omega_\psi$), the Galois group of the maximal abelian extension of $L_\infty$ unramified outside $\Sigma_p$.
\begin{prop}
\label{prop:compareellipticselmertox}
\begin{itemize}
\item[(i)] $\textup{Sel}^{\,\prime}_\varpi(A/F)\hookrightarrow \textup{Sel}^{\,\prime}_\varpi(A/F_\infty)^\Gamma$.
\item[(ii)] $\textup{Sel}^{\,\prime}_\varpi(A/F_\infty) \subset H^1_{\FF^*_{-}}(F,\TT^*)$.
\end{itemize}
\end{prop}
\begin{proof}
Both assertions are standard. Our non-anomaly assumption (\ref{eqn:hna}) shows that
 $$H^1(F_\infty/F,A[\varpi^\infty])=0$$
 (c.f. the proof of \cite[Lemma 2.2]{rubin87sha}) and the inflation/restriction sequence shows that the restriction map
$$H^1(F,A[\varpi^\infty])\lra H^1(F_\infty,A[\varpi^\infty])$$
is injective. This proves (i). The assertion (ii) may be proved following the second half of the proof  of \cite[Theorem 12]{coatesinfinitedescent} \emph{verbatim},  and using the identification
$$\textup{Hom}(\frak{X}_\infty,A[\varpi^\infty])\cong H^1_{\FF^*_{-}}(F,\TT^*),$$
which is a consequence of the twisting formalism in~\cite[\S6]{r00}.

%\begin{itemize}
%We observe that,
%\\\\

%\item for a place $\lambda \nmid p$ of $F$,
%\begin{align*}
%H^1_{\FF_{\textup{BK}}}(F_\lambda,T)&\supset H^1_{\textup{ur}}(F_\lambda,T)\\
%&\supset \textup{im}\left(H^1(F_\lambda, T\otimes\LL^{\textup{cyc}})\ra H^1(F_\lambda, T)\right)\\
%&\supset \textup{im}\left(H^1(F_\lambda, T\otimes\LL)\ra H^1(F_\lambda, T)\right)\\
%&=: H^1_{\FF_{-}}(F_\lambda,T),
%\end{align*}
%where the first inclusion follows from \cite[Lemma I.3.5(ii)]{r00} and second from \cite[Corollary B.3.4]{r00}. We thus conclude on the dual side that $H^1_{\FF_{\textup{BK}}^*}(F,T^*) \subset H^1_{\FF_-^*}(F,T^*)$
%\item The proof of \cite[Lemma 6]{coatesinfinitedescent} (applied with $A[\varpi^\infty]$ in place of $E[\frak{p}^\infty]$ in loc.cit.) shows that $H^i(F_\infty/F,T^*)=0$ for $i\geq 1$. This in turn implies that
%$$H^1_{\FF_{BK}^\bullet}(F_\frak{p},T^*)=H^1_{\FF_-^*}(F_{\frak{p}},T^*)$$
%for every prime $\frak{p}$ of $F$ above $p$, where
%$$H^1_{\FF_{BK}^\bullet}(F_\frak{p},T^*):=$$
%for any finite subextension $\mathcal{M}$ of $F_\infty$ and a prime $\frak{p}$ of $\mathcal{M}$ above $p$, we have
%$$H^1_{\FF_{\textup{BK}}}(\mathcal{M}_\frak{p},T_\chi)=\mathcal{U}_{L_\chi M}^\chi=H^1(\mathcal{M}_\frak{p},T_\chi),$$
%where $\mathcal{U}_{L_\chi M}^\chi$ are the $\chi$-part of the semi-local units of $L_\chi M$ at $\frak{p}$ of $M$ above $p$, and where the last equality holds thanks to the assumption (\ref{eqn:hna});
%\end{itemize}
\end{proof}
Let $g_\chi \in \textup{char}\left(H^1_{\FF^*_{-}}(F,\TT^*)^\vee\right)$ be any generator. Assuming Conjecture~\ref{conj:yager}, we have $g_\chi=u\cdot \al_\chi^\Sigma$ for a unit $u \in \LL_\mathcal{W}^\times$.

\begin{thm}\label{thm:mainabvar} Assume the truth of Conjecture~\ref{conj:yager} and the hypotheses of Theorem~\ref{thm:mainconjforTchi}.
\begin{itemize}
\item[(i)]  For $\alpha_\varpi=\textup{ord}_\varpi(\psi(g_\chi))$, we have $| H^1_{\FF_-^*}(F,T^*)| = p^{\alpha_\wp}$.
\item[(ii)] Suppose $L(\psi_\varepsilon,0)\neq0$. Then $A(F)$ is finite and $\Sha_{A/F}[\varpi^\infty]$ is finite.
\end{itemize}
\end{thm}

\begin{proof}
Let $A_\psi=\ker\left\{\LL\stackrel{\gamma\mapsto\psi(\gamma)}{\lra} \ooo\right\}$. A variant of the proof of the identity (\ref{eqn:rhsmodA}) shows that
$$[\ooo:\textup{char}\left(H^1_{\FF^*_{-}}(F,(T_\chi\otimes\LL)^*)^\vee\right)\otimes_{\LL}\LL/\mathcal{A_\psi}]=| H^1_{\FF_-^*}(F,T^*)|$$
and that
\begin{align*}
[\ooo:\textup{char}\left(H^1_{\FF^*_{-}}(F,(T_\chi\otimes\LL)^*)^\vee\right)\otimes_{\LL}\LL/\mathcal{A_\psi}]&=[\ooo:\psi\left(\textup{char}(H^1_{ \FF^*_{-}}(F,(T_\chi\otimes\LL)^*)^\vee)\right)]\\
&=[\ooo:\psi(g_\chi)]
\end{align*}
This proves (i).

The assertion (ii) follows immediately from (i), Proposition~\ref{prop:compareellipticselmertox} and  the interpolation property of the $p$-adic $L$-function $\al_\chi^\Sigma$ (which may be found in \cite[Theorem II]{ht93}).
\end{proof}

%\section{Linear Algebra}
%\label{appendix:linearalgebra}
%%%%%%%%%%%%%%%%%%%%%%%%%%%%%%%%%%%%%%%%%%%%%%%%
{\scriptsize
\bibliographystyle{halpha}
\bibliography{references}

\begin{thebibliography}{B{\"u}y13b}

\bibitem[AH06]{agboolahoward}
Adebisi Agboola and Benjamin Howard.
\newblock Anticyclotomic {I}wasawa theory of {CM} elliptic curves.
\newblock {\em Ann. Inst. Fourier (Grenoble)}, 56(4):1001--1048, 2006.

\bibitem[Arn07]{arnoldCMform}
Trevor Arnold.
\newblock Anticyclotomic main conjectures for {CM} modular forms.
\newblock {\em J. Reine Angew. Math.}, 606:41--78, 2007.

\bibitem[Arn10]{arnoldhida}
Trevor Arnold.
\newblock Hida families, {$p$}-adic heights, and derivatives.
\newblock {\em Ann. Inst. Fourier (Grenoble)}, 60(6):2275--2299, 2010.

\bibitem[BK90]{bk}
Spencer Bloch and Kazuya Kato.
\newblock {$L$}-functions and {T}amagawa numbers of motives.
\newblock In {\em The Grothendieck Festschrift, Vol.\ I}, volume~86 of {\em
  Progr. Math.}, pages 333--400. Birkh\"auser Boston, Boston, MA, 1990.

\bibitem[B{\"u}y09a]{kbbstark}
K\^az{\i}m B{\"u}y\"ukboduk.
\newblock Kolyvagin systems of {S}tark units.
\newblock {\em J. Reine Angew. Math.}, 631:85--107, 2009.

\bibitem[B{\"u}y09b]{kbbiwasawa}
K\^az{\i}m B{\"u}y\"ukboduk.
\newblock Stark units and the main conjectures for totally real fields.
\newblock {\em Compositio Math.}, 145:1163--1195, 2009.

\bibitem[B{\"u}y10]{kbbesrankr}
K{\^a}zim B{\"u}y{\"u}kboduk.
\newblock On {E}uler systems of rank {$r$} and their {K}olyvagin systems.
\newblock {\em Indiana Univ. Math. J.}, 59(4):1277--1332, 2010.

\bibitem[B{\"u}y11]{kbb}
K\^az{\i}m B{\"u}y\"ukboduk.
\newblock {$\Lambda$}-adic {K}olyvagin systems.
\newblock {\em IMRN}, 2011(14):3141--3206, 2011.

\bibitem[B{\"u}y12]{kbbdeform}
K\^az{\i}m B{\"u}y\"ukboduk.
\newblock Deformations of {K}olyvagin systems, 2012.
\newblock Submitted.

\bibitem[B{\"u}y13a]{kbbstarksupersingular}
K\^az{\i}m B{\"u}y\"ukboduk.
\newblock {O}n the {I}wasawa theory of {CM} fields for supersingular primes,
  2013.
\newblock 26pp., Preprint.

\bibitem[B{\"u}y13b]{kbbanticycloK}
K\^az{\i}m B{\"u}y\"ukboduk.
\newblock {R}ubin-{S}tark elements and anticyclotomic main conjectures for {CM}
  elliptic curves, 2013.
\newblock In preparation.

\bibitem[Coa83]{coatesinfinitedescent}
John Coates.
\newblock Infinite descent on elliptic curves with complex multiplication.
\newblock In {\em Arithmetic and geometry, {V}ol. {I}}, volume~35 of {\em
  Progr. Math.}, pages 107--137. Birkh\"auser Boston, Boston, MA, 1983.

\bibitem[dS87]{deshalit87}
Ehud de~Shalit.
\newblock {\em Iwasawa theory of elliptic curves with complex multiplication},
  volume~3 of {\em Perspectives in Mathematics}.
\newblock Academic Press Inc., Boston, MA, 1987.
\newblock $p$-adic $L$ functions.

\bibitem[Hid06]{hida06}
Haruzo Hida.
\newblock Anticyclotomic main conjectures.
\newblock {\em Doc. Math.}, (Extra Vol.):465--532 (electronic), 2006.

\bibitem[Hid09]{hidaquadratic}
Haruzo Hida.
\newblock Quadratic exercises in {I}wasawa theory.
\newblock {\em Int. Math. Res. Not. IMRN}, (5):912--952, 2009.

\bibitem[Hsi12]{hsiehCMmainconj}
Ming-Lun Hsieh.
\newblock Iwasawa main conjecture for {CM} fields, 2012.
\newblock 81pp., Preprint.

\bibitem[HT93]{ht93}
H.~Hida and J.~Tilouine.
\newblock Anti-cyclotomic {K}atz {$p$}-adic {$L$}-functions and congruence
  modules.
\newblock {\em Ann. Sci. \'Ecole Norm. Sup. (4)}, 26(2):189--259, 1993.

\bibitem[HT94]{ht94}
H.~Hida and J.~Tilouine.
\newblock On the anticyclotomic main conjecture for {CM} fields.
\newblock {\em Invent. Math.}, 117(1):89--147, 1994.

\bibitem[Kat78]{katz78}
Nicholas~M. Katz.
\newblock {$p$}-adic {$L$}-functions for {CM} fields.
\newblock {\em Invent. Math.}, 49(3):199--297, 1978.

\bibitem[Mai08]{mainardi}
Fabio Mainardi.
\newblock On the main conjecture for {CM} fields.
\newblock {\em Amer. J. Math.}, 130(2):499--538, 2008.

\bibitem[MR04]{mr02}
Barry Mazur and Karl Rubin.
\newblock Kolyvagin systems.
\newblock {\em Mem. Amer. Math. Soc.}, 168(799):viii+96, 2004.

\bibitem[Nek06]{nekovar06}
Jan Nekov{\'a}{\v{r}}.
\newblock Selmer complexes.
\newblock {\em Ast\'erisque}, (310):viii+559, 2006.

\bibitem[NSW08]{neukirch}
J{\"u}rgen Neukirch, Alexander Schmidt, and Kay Wingberg.
\newblock {\em Cohomology of number fields}, volume 323 of {\em Grundlehren der
  Mathematischen Wissenschaften [Fundamental Principles of Mathematical
  Sciences]}.
\newblock Springer-Verlag, Berlin, second edition, 2008.

\bibitem[Och05]{ochiaideform}
Tadashi Ochiai.
\newblock Euler system for {G}alois deformations.
\newblock {\em Ann. Inst. Fourier (Grenoble)}, 55(1):113--146, 2005.

\bibitem[Pop04]{popescu}
Cristian~D. Popescu.
\newblock {Rubin's integral refinement of the Abelian Stark conjecture.}
\newblock {Providence, RI: American Mathematical Society (AMS)}, 2004.

\bibitem[PR98]{pr-es}
Bernadette Perrin-Riou.
\newblock Syst\`emes d'{E}uler {$p$}-adiques et th\'eorie d'{I}wasawa.
\newblock {\em Ann. Inst. Fourier (Grenoble)}, 48(5):1231--1307, 1998.

\bibitem[Rib76]{ribetcompositio}
Kenneth~A. Ribet.
\newblock {Dividing rational points on Abelian varieties of CM-type.}
\newblock {\em Compos. Math.}, 33:69--74, 1976.

\bibitem[Rub87]{rubin87sha}
Karl Rubin.
\newblock Tate-{S}hafarevich groups and {$L$}-functions of elliptic curves with
  complex multiplication.
\newblock {\em Invent. Math.}, 89(3):527--559, 1987.

\bibitem[Rub91]{rubinmainconj}
Karl Rubin.
\newblock The ``main conjectures'' of {I}wasawa theory for imaginary quadratic
  fields.
\newblock {\em Invent. Math.}, 103(1):25--68, 1991.

\bibitem[Rub92]{ru92}
Karl Rubin.
\newblock Stark units and {K}olyvagin's ``{E}uler systems''.
\newblock {\em J. Reine Angew. Math.}, 425:141--154, 1992.

\bibitem[Rub96]{ru96}
Karl Rubin.
\newblock A {S}tark conjecture ``over {$\bold Z$}'' for abelian {$L$}-functions
  with multiple zeros.
\newblock {\em Ann. Inst. Fourier (Grenoble)}, 46(1):33--62, 1996.

\bibitem[Rub00]{r00}
Karl Rubin.
\newblock {\em Euler systems}, volume 147 of {\em Annals of Mathematics
  Studies}.
\newblock Princeton University Press, Princeton, NJ, 2000.
\newblock Hermann Weyl Lectures. The Institute for Advanced Study.

\bibitem[Shi75]{shimura75}
Goro Shimura.
\newblock On some arithmetic properties of modular forms of one and several
  variables.
\newblock {\em Ann. of Math. (2)}, 102(3):491--515, 1975.

\bibitem[ST68]{serretate}
Jean-Pierre Serre and John Tate.
\newblock Good reduction of abelian varieties.
\newblock {\em Ann. of Math. (2)}, 88:492--517, 1968.

\bibitem[Wil95]{wiles}
Andrew Wiles.
\newblock Modular elliptic curves and {F}ermat's last theorem.
\newblock {\em Ann. of Math. (2)}, 141(3):443--551, 1995.

\bibitem[Yag82]{yager}
Rodney~I. Yager.
\newblock On two variable {$p$}-adic {$L$}-functions.
\newblock {\em Ann. of Math. (2)}, 115(2):411--449, 1982.

\end{thebibliography}
}
\end{document}